\documentclass[a4paper]{article}
\usepackage[a4paper, total={6in, 9in}]{geometry}
\usepackage{graphicx,multicol,float,amsmath,amsthm,amsfonts,dsfont,amssymb,amsthm,MnSymbol,stmaryrd,color,url}
\usepackage[all,2cell, cmtip]{xy} \UseAllTwocells \SilentMatrices
\usepackage[hidelinks]{hyperref}
\usepackage[titletoc]{appendix}
\usepackage[bbgreekl]{mathbbol}
\usepackage{verbatim}
\usepackage{mathtools}
\usepackage{eucal}

\newtheorem{thm}{Theorem}[section]
\newtheorem{thmintro}{Theorem}

\newtheorem{lem}[thm]{Lemma}
\newtheorem{prop}[thm]{Proposition}
\newtheorem{cor}[thm]{Corollary}

\theoremstyle{definition}
\newtheorem{defn}[thm]{Definition}
\theoremstyle{remark}

\newtheorem{remark}[thm]{Remark}
\newtheorem{exam}[thm]{Example}

\newtheoremstyle{outlinenotes}{}{}{\color{blue}}{}{\color{blue}\bfseries}{}{ }{}
\theoremstyle{outlinenotes}

\newtheoremstyle{todonotes}{}{}{\color{red}}{}{\color{red}\bfseries}{}{ }{}
\theoremstyle{todonotes}

\newcommand{\Deltaac}{\Delta_{\rm ac}}
\newcommand{\Deltain}{\Delta_{\rm in}}
\newcommand{\Simp}{\mathbb{\Delta}}
\newcommand{\Simpop}{\Simp^{\rm op}}
\newcommand{\Simpac}{\Simp_{\rm ac}}
\newcommand{\Simpin}{\Simp_{\rm in}}

\newcommand{\Fin}{{\rm Fin}}
\newcommand{\cFin}{\mathcal{F}{\rm in}}
\newcommand{\sSetf}{\Fin_\Delta}
\newcommand{\sSetfop}{(\sSetf)^{\rm op}}
\newcommand{\csSetf}{\cFin_\Simp}
\newcommand{\csSetfop}{(\csSetf)^{\rm op}}
\newcommand{\tSetf}{\Fin_{\Delta^2}}
\newcommand{\tSetfop}{(\tSetf)^{\rm op}}
\newcommand{\ctSetf}{\cFin_{\Simp^2}}
\newcommand{\ctSetfop}{(\ctSetf)^{\rm op}}
\newcommand{\Finpt}{\Fin_*}

\newcommand{\cFinpt}{\cFin_*}
\newcommand{\FinDel}{\Finpt\times\Delta^{\rm op}}
\newcommand{\cFinDel}{\cFinpt \times \Simpop}
\newcommand{\Cat}{{\rm Cat}}

\newcommand{\qCat}{{\rm qCat}}
\newcommand{\Catoo}{\mathcal{C}{\rm at}_\infty}

\newcommand{\Bicatoo}{\mathcal{C}{\rm at}_{\inftwo}}
\newcommand{\smBicatoo}{\Bicatoo^\otimes}
\newcommand{\Alg}{{\rm Alg}}
\newcommand{\cAlg}{\mathcal{A}{\rm lg}}
\newcommand{\csmAlg}{\cAlg^\amalg}
\newcommand{\Coalg}{{\rm Coalg}}
\newcommand{\cCoalg}{\mathcal{C}{\rm oalg}}
\newcommand{\csmCoalg}{\cCoalg^\amalg}
\newcommand{\Bialg}{{\rm Bialg}}

\newcommand{\cBialg}{\mathcal{B}{\rm ialg}}
\newcommand{\csmBialg}{\cBialg^\amalg}
\newcommand{\Fun}{{\rm Fun}}
\newcommand{\cFun}{\mathcal{F}{\rm un}}
\newcommand{\Map}{{\rm Map}}
\newcommand{\cMap}{\mathcal{M}{\rm ap}}

\newcommand{\cArr}[1]{\mathcal{A}\mathrm{rr}_{{#1}}}

\newcommand{\ctSeg}[1]{2\mathcal{S}\mathrm{eg}({#1})}
\newcommand{\apExact}{\mathrm{pExact}_+}
\newcommand{\capExact}{\mathrm{p}\mathcal{E}\mathrm{xact}_+}

\def\cSpan#1{\mathcal{S}{\rm pan}\left({#1}\right)}

\def\csmtSpan#1{\mathcal{S}{\rm pan}_2^\times \, {#1}}

\def\Pyr#1{\Sigma^{#1}}
\def\cPyr#1{\mathbb{\Sigma}^{#1}}
\def\Wedge#1{\Lambda^{#1}}
\def\cWedge#1{\mathbb{\Lambda}^{#1}}
\def\Pyrnc#1{\doublewedge^{#1}}
\def\Cart#1{\Pi{#1}}
\def\cCart#1{\mathbb{\Pi}#1}

\def\Cartop#1{\Pi{#1}^{\rm op}}
\def\cCartop#1{\mathbb{\Pi}#1^{\rm op}}
\def\Sing#1{P{#1}}
\def\cSing#1{\mathcal{P}{#1}}
\def\Singop#1{P{#1}^{\rm op}}
\def\cSingop#1{\mathcal{P}{#1}^{\rm op}}

\def\Unstr#1{\left\langle{#1}\right\rangle}
\def\sp#1{{\rm sp}\left(#1\right)}

\newcommand{\cA}{\mathcal{A}}
\newcommand{\cB}{\mathcal{B}}

\newcommand{\fB}{\mathfrak{B}}

\newcommand{\cC}{\mathcal{C}}

\newcommand{\cD}{\mathcal{D}}
\newcommand{\cE}{\mathcal{E}}

\newcommand{\cF}{\mathcal{F}}

\newcommand{\fG}{\mathfrak{G}}

\newcommand{\cJ}{\mathcal{J}}

\newcommand{\cM}{\mathcal{M}}

\newcommand{\cN}{\mathcal{N}}

\newcommand{\cS}{\mathcal{S}}

\newcommand{\cX}{\mathcal{X}}
\newcommand{\cY}{\mathcal{Y}}

\newcommand{\inftwo}{(\infty,2)}
\newcommand\Sp{\mathcal{S}{\rm p}}
\newcommand\Un{{\rm Un}}
\newcommand{\actmorR}{\rightarrow\Mapsfromchar}
\newcommand{\actmorL}{\Mapstochar\leftarrow}
\newcommand\inmorR{\rightarrowtail}
\newcommand\inmorL{\leftarrowtail}
\newcommand\op{{\rm op}}

\def\Simplex#1{\Delta\left[{#1}\right]}

\def\Simplexn#1{\nabla\left[{#1}\right]}

\newcommand{\id}{{\rm Id}}
\renewcommand{\lim}{\operatornamewithlimits{lim}}

\newcommand{\adjRelayII}[3][2.2em]{\ensuremath{\SelectTips{cm}{10}\xymatrix@C=#1@1{{#2} \ar@<1ex>[r]^-{\ArgI}^-{}="1" & {#3} \ar@<1ex>[l]^-{\ArgII}^-{}="2" \ar@{}"1";"2"|(.3){\hbox{}}="7" \ar@{}"1";"2"|(.7){\hbox{}}="8" \ar@{|-} "8" ;"7"}}}

\newcommand{\radjRelayII}[3][2.2em]{\ensuremath{\SelectTips{cm}{10}\xymatrix@C=#1@1{{#2} \ar@<-1ex>[r]_-{\ArgI}^-{}="1" & {#3} \ar@<-1ex>[l]_-{\ArgII}^-{}="2" \ar@{}"1";"2"|(.3){\hbox{}}="7" \ar@{}"1";"2"|(.7){\hbox{}}="8" \ar@{|-} "7" ;"8"}}}

\numberwithin{equation}{section}

\title{The universal Hall bialgebra of a double $2$-Segal space}
\author{Mark D Penney\footnote{email: \href{mailto:mpenney@mpim-bonn.mpg.de}{mpenney@mpim-bonn.mpg.de}}}
\date{}

\begin{document}
 \maketitle
 \begin{abstract}
Hall algebras and related constructions have had diverse applications in mathematics and physics, ranging from representation theory and quantum groups to Donaldson-Thomas theory and the algebra of BPS states. The theory of $2$-Segal spaces was introduced independently by Dyckerhoff--Kapranov and G\'alvez-Carrillo--Kock--Tonks as a unifying framework for Hall algebras: every $2$-Space defines an algebra in the $\infty$-category of spans, and different Hall algebras correspond to different linearisations of this universal Hall algebra.

A recurring theme is that Hall algebras can often be equipped with a coproduct which makes them a bialgebra, possibly up to a `twist'. In this paper will explain the appearance of these bialgebraic structures using the theory of $2$-Segal spaces: We construct the universal Hall bialgebra of a double $2$-Segal space, which is a lax bialgebra in the $\inftwo$-category of bispans. Moreover, we show how examples of double $2$-Segal spaces arise from Waldhausen's $S$-construction.
 \end{abstract}
 
 \tableofcontents
 

\section{Introduction}
\label{intro}

Dyckerhoff--Kapranov \cite{DK12} and G\'alvez-Carrillo--Kock--Tonks \cite{KockI} have independently introduced the notion of a {\em $2$-Segal space} as a unifying perspective on the various Hall algebra constructions that appear in the literature. The idea is that every $2$-Segal space $\cX$ defines an algebra object in the $\infty$-category of spans called the {\em universal Hall algebra} of $\cX$. Then different flavours of Hall algebras correspond to different choices in how to linearise this universal Hall algebra. For example, the Ringel-Hall algebra linearises via locally constant functions \cite{Ringel}, To\"en's derived Hall algebras via homotopy cardinality \cite{toen2006derived}, Joyce's motivic Hall algebras via stack functions \cite{joyce2007configurations} and Lusztig's categorification of the positive part of quantum groups via perverse sheaves \cite{lusztig1991quivers}.

A recurring theme is that the various Hall algebras can be equipped with a coproduct which makes them a bialgebra, possibly up to a `twist'. The canonical example of this is Green's theorem for the Ringel-Hall algebra \cite{Green}, but similar results appear in the work of Joyce and Lusztig. In this paper we will show how to explain the appearance of these bialgebraic structures using the theory of $2$-Segal spaces: We construct the {\em universal Hall bialgebra of a double $2$-Segal space}, which is a lax bialgebra in the $\inftwo$-category of bispans. 

Our approach draws significant inspiration from Dyckerhoff's proof of Green's theorem in the language of $2$-Segal spaces \cite{D15}: Waldhausen's $S$-construction applied to an abelian category $A$ produces a $2$-Segal space $S_\bullet A: \Simpop \to \cS$ with $S_n A$ the moduli space of length $n$ flags in $A$. The universal Hall algebra of $A$ is the algebra object in the $\infty$-category of spans whose underlying space is $S_1 A$ with product and unit given by
\begin{equation}
\label{IntroHallprod}
 \xymatrixrowsep{.3pc} \xymatrixcolsep{.7pc} \xymatrix{ & & 0\to a \to b \to c\to 0 \ar@{|->}[rrddd] \ar@{|->}[llddd] & & & & & 0 \ar@{|->}[rrddd] & &  \\
 & & S_2 A \ar[rdd] \ar[ldd] & & & & & S_0 A \ar[ldd] \ar[rdd]& & \\
 & & & & & & & & & \\
 (a,c) & S_1 A^{ 2} & & S_1 A & b & & \ast & & S_1 A & 0\, .}
\end{equation}
Dually, these diagrams can be read from right to left, making $S_1 A$ a coalgebra object in the $\infty$-category of spans which we call the universal Hall coalgebra of $A$.

The operations $\Delta \mu$ and $\overline{\mu}^2 \Delta^2$ are implemented, respectively, by the spans
\begin{equation*}
 \xymatrixrowsep{.1pc} \xymatrixcolsep{2pc} \xymatrix{ & S_+ A \ar[rdd] \ar[ldd] &  &  & & S_\Box A \ar[rdd] \ar[ldd] & \\
 & & & {\rm and} & & & & \\
 S_1 A^{2} & & S_1 A^{2} & & S_1 A^{2} & & S_1 A^{2} \, ,}
\end{equation*}
where $S_+ A$ and $S_\Box A$  are, respectively, the moduli spaces of diagrams in $A$ of the form
\begin{equation*}
 \xymatrixrowsep{.5pc} \xymatrixcolsep{1.1pc} \xymatrix{ & b \ar[d] & & & a \ar[r] \ar[d] & b \ar[r] & c \ar[d]  \\
 a' \ar[r] & b' \ar[d] \ar[r] & c' & \mathrm{and} & a'\ar[d] & & c' \ar[d] \\
 & b'' & & & a'' \ar[r] & b'' \ar[r] & c'' }
\end{equation*}
having each row and column a short exact sequence. To investigate the compatibility of $\mu$ and $\Delta$, consider the diagram
\begin{equation}
\label{laxbialginf}
 \xymatrixrowsep{.7pc} \xymatrixcolsep{1.5pc} \xymatrix{ & S_+ A \ar[rd] \ar[ld] & \\
 S_1 A^2 & S_{2,2} A \ar[r] \ar[d] \ar[u] \ar[l] & S_1 A^2\, , \\
 & S_\Box A \ar[ru] \ar[lu] & }
\end{equation}
where $S_{2,2} A$ is the moduli spaces of diagrams in $A$ of the form
\begin{equation*}
 \xymatrixrowsep{.5pc} \xymatrixcolsep{1.3pc} \xymatrix{ a \ar[r] \ar[d] & b \ar[d] \ar[r] & c \ar[d] \\
  a'\ar[d] \ar[r] &b' \ar[r] \ar[d] & c' \ar[d] \\
 a'' \ar[r] & b'' \ar[r] & c'' }
\end{equation*}
having each row and column a short exact sequence. The morphism $S_{2,2} A \to S_+ A$ is an equivalence since every cross of short exact sequences can be completed to a $3$x$3$ grid of short exact sequences in an essentially unique way using pullbacks, pushouts, kernels and cokernels. Unfortunately, the morphism $S_{2,2} A \to S_\Box A$ is an equivalence if and only if $A$ is the trivial abelian category (\cite{D15} 2.38). 

We are therefore forced to conclude that $S_1 A$ {\em is not}, in general, a bialgebra. Nonetheless, the diagram in Eq.~\ref{laxbialginf} {\em does} define a non-invertible $2$-morphism
\begin{equation*}
 \Delta \mu \Rightarrow \overline{\mu}^2 \Delta^2
\end{equation*}
in the $\inftwo$-category of bispans. In other words, the $2$-morphism $\Delta \mu \Rightarrow \overline{\mu}^2 \Delta^2$ is a witness for the {\em lax compatibility} of the universal Hall algebra and coalgebra of $A$. These structures taken together define a {\em lax bialgebra object} in the $\inftwo$-category of bispans which we call the universal Hall bialgebra of the abelian category $A$.

In defining the $2$-morphism witnessing the lax compatibility in Eq.~\ref{laxbialginf} we made use of the space $S_{2,2} A$ which cannot be directly built from the $2$-Segal space $S_\bullet A$. Indeed, the additional structure we needed was that Waldhausen's $S$-construction can be iterated to define a bisimplicial space $S_{\bullet,\bullet} A$ where $S_{n,k} A$ is the moduli space of length $n$ flags of length $k$ flags in $A$. The bisimplicial space $S_{\bullet,\bullet} A$ is a particular example of a {\em double $2$-Segal space}: a bisimplicial space $\cX_{\bullet,\bullet}: (\Simpop)^2 \to \cS$ such that $\cX_{\bullet,k}$ and $\cX_{k,\bullet}$ are $2$-Segal spaces for each $k$.

A wealth of examples of double $2$-Segal spaces arise from the $S$-construction of so-called {\em augmented proto-exact $\infty$-categories}: generalisations of exact $\infty$-categories \cite{barwick2015exact} which are not necessarily additive or even pointed. There are two variations on the $S$-construction which define double $2$-Segal spaces. The first is the {\em iterated $S$-construction} $S_{\bullet,\bullet} \cA$ of an augmented proto-exact $\infty$-category which was described above for abelian categories. The second is the {\em monoidal $S$-construction} $S^\otimes_{\bullet,\bullet} \cA$ of a monoidal augmented proto-exact $\infty$-category.
\begin{thmintro}
\label{Thm:MainSConst}
 Both the iterated $S$-construction of an augmented proto-exact $\infty$-category and the monoidal $S$-construction of a monoidal augmented proto-exact $\infty$-category are double $2$-Segal spaces. 
\end{thmintro}

The universal Hall bialgebra of an abelian category outlined above is a special case of the main result of this paper:
\begin{thmintro}
 \label{Thm:MainSumm}
 The space of $(1,1)$-simplices $\cX_{1,1}$ of a double $2$-Segal space $\cX_{\bullet,\bullet}$ is canonically a lax bialgebra in the $\inftwo$-category of bispans. We call this the universal Hall bialgebra of $\cX$. 
\end{thmintro}

Our construction of the universal Hall bialgebra leverages the fact that the initial bisimplicial object in an $\infty$-category having finite limits is
\begin{equation*}
 \Simplex{\bullet,\bullet}: (\Simpop)^2 \to \ctSetfop,
\end{equation*}
where $\ctSetf$ is the nerve of the category of level-wise finite bisimplicial sets: any bisimplicial object $\cX_{\bullet,\bullet} \in \cC_{\Simp^2}$ in an $\infty$-category $\cC$ having finite limits defines a finite limit preserving functor $\ctSetfop \to \cC$ by right Kan extension
\begin{equation*}
 \xymatrixrowsep{1.1pc} \xymatrix{ \ctSetfop \ar[r] & \cC \\
 (\Simpop)^2 \ar@{^{(}->}[u] \ar[ru]_-\cX & }
\end{equation*}
which sends $\Simplex{\bullet,\bullet}$ to $\cX_{\bullet,\bullet}$. The structures underlying the universal Hall bialgebra of Theorem \ref{Thm:MainSumm} are directly inherited from the following:
\begin{thmintro}
 \label{Thm:Maincomb}
 The standard $(1,1)$-simplex $\Simplex{1,1}$ is a totally lax bialgebra in the $\inftwo$-category of bispans in $\ctSetfop$. That is, it carries a laxly associative product and lax coassociative coproduct which are laxly compatible.
\end{thmintro}
One of the key benefits of this approach is that it allows us to define the universal Hall bialgebra of a double $2$-Segal object in an arbitrary $\infty$-category having finite limits, such as $\infty$-categories of spaces or (derived) stacks.

This paper is the second in a series of three papers whose aim it is to define new examples of bimonoidal categories using $2$-Segal spaces. The first paper \cite{penney2017simplicial} gave a combinatorial construction of the universal Hall algebra of a $2$-Segal object, in the process laying much of the technical groundwork for this paper. The third paper \cite{penney2017bimon} defines the Hall bimonoidal category of a double $2$-Segal space as a linearisation via local systems of the universal Hall bialgebra defined in this paper. 

\paragraph{Outline.} We begin in Section \ref{Sec:catbackground} with the technical background necessary to define the notion of lax bialgebra objects in $\csmtSpan{\cC}$, the symmetric monoidal $\inftwo$-category of bispans in $\cC$. In short, they are defined as certain symmetric monoidal lax functors $\csmBialg \rightsquigarrow \csmtSpan{\cC}$, where $\csmBialg$ is the $\infty$-category which corepresents bialgebras. In Section \ref{Sec:laxfun} we review symmetric monoidal $\inftwo$-categories and symmetric monoidal lax functors between them. Then in Section \ref{Sec:bispans} we introduce the twisted arrow construction which appears throughout this paper and review results from \cite{penney2017simplicial} which present $\csmtSpan{\cC}$ in a convenient form for our purposes. Section \ref{Sec:catbackground} ends with the definition of $\csmBialg$ and the definition of lax bialgebra objects.

As discussed above the protagonists in this paper are double $2$-Segal spaces, the subjects of Section \ref{Sec:double2Seg}. Given the definition of a $2$-Segal space, which we review in Section \ref{Sec:Rev2Seg}, the definition of a double $2$-Segal space is straightforward. Moreover, one needs to develop essentially no theoretical machinery concerning double $2$-Segal spaces to define their universal Hall bialgebras. Instead, the purpose of this Section is to construct examples of double $2$-Segal spaces using the $S$-construction of augmented proto-exact $\infty$-categories as outlined in Theorem \ref{Thm:MainSConst}. This is done in Section \ref{Sec:augexact}. 

Section \ref{Sec:totlax} contains the main construction of the paper: we show that every bisimplicial object $\cX \in \cC_{\Simp^2}$ equips its space of $(1,1)$-simplices $\cX_{1,1}$ with the structure of a totally lax bialgebra in $\csmtSpan{\cC}$. This is, as we explain in Section \ref{Sec:intospacesBi}, an essentially formal consequence of Theorem \ref{Thm:Maincomb}. The latter construction, namely the endowing of $\Simplex{1,1}$ with a totally lax bialgebra structure as an object of $\csmtSpan{\ctSetfop}$, occupies Section \ref{Sec:bialgcomb}. 

Finally, we conclude the paper in Section \ref{Sec:SegbiAssoc} with the definition of the universal Hall bialgebra of a double $2$-Segal object and the proof that it is a lax bialgebra. 

\paragraph{Acknowledgements.} The majority of this paper was part of the author's DPhil thesis under the supervision of Christopher L. Douglas. We are indebted to Joachim Kock for his thorough feedback as an external thesis examiner. Finally, we would also like to thank Tobias Dyckerhoff for a number of insightful conversations concerning $2$-Segal spaces and for pointing out an error in an early version of these results. 

\paragraph{Notational conventions and simplicial preliminaries.} By an $\infty$-category we will always mean a quasi-category. As such, we crucially rely upon the theory of $\infty$-categories developed by \cite{JoyalJPAA, joyal2008theory} and Lurie \cite{HTT,LurieHA}. 

To distinguish ordinary from $\infty$- categories, we use Greek letters (e.g. $\Delta$) or ordinary font (e.g. $C$) to denote the former and blackboard Greek letter (e.g. $\Simp$) or calligraphic font (e.g. $\cC$) to denote the latter.

The category of functors between between ordinary categories is denoted $\Fun(-,-)$ while for $\infty$-categories it is $\cFun(-,-)$. Similarly, $\Map(-,-)$ is the groupoid of functors and natural isomorphisms and $\cMap(-,-)$ is the largest Kan complex inside $\cFun(-,-)$. 

The $\infty$-category of $\infty$-categories, $\Catoo$, is defined to be the simplicial nerve of $\qCat$, which has objects quasi-categories and mapping spaces given by $\cMap(-,-)$ (\cite{HTT} 3.0.0.1). The $\infty$-category of spaces, $\cS$, is the full subcategory of $\Catoo$ on those quasi-categories which are Kan complexes. The inclusion $\xymatrixcolsep{.8pc}\xymatrix{\cS \ar@{^{(}->}[r] &  \Catoo}$ admits a right adjoint $(-)^\simeq$ (\cite{HTT} 1.2.5.3) and a left adjoint $(-)^\mathrm{gpd}$ (\cite{HTT} 1.2.5.6).

The category $\Fin$ is the category of finite sets with every object isomorphic to one of the form
\begin{equation*}
 \underline{n}=\{1,\cdots, n\} \in \Fin.
 \end{equation*}
Similarly, $\Finpt$ is the category of finite pointed sets, the objects of which are denoted $X_\ast$ where $\ast$ is the basepoint and $X$ the complement. The $\infty$-categories $\cFin$ and $\cFinpt$ are, respectively, the nerves of $\Fin$ and $\Finpt$.

The objects of the category $\Delta$ of non-empty finite linear orders are
\begin{equation*}
 [n] = \{0<1<\ldots<n \} \in \Delta.
\end{equation*}
There are two distinguished classes of morphisms in $\Delta$: the {\em active} morphisms, denoted $\actmorR$, which preserve endpoints and the {\em inert} morphisms, denoted $\inmorR$, which are inclusions of subintervals. Every morphism in $\Delta$ can be uniquely factored as an active followed by an inert morphism. We denote the wide subcategories of, respectively, active and inert morphisms by $\Deltaac$ and $\Deltain$. The $\infty$-categories $\Simp$, $\Simpac$ and $\Simpin$ are, respectively, the nerves of $\Delta$, $\Deltaac$ and $\Deltain$.

A {\em $k$-fold simplicial object} in an ordinary category $C$ or $\infty$-category $\cC$ is an object, respectively, of
 \begin{equation*}
  C_{\Delta^k} = \Fun\left(\left(\Delta^\op\right)^{\times k}, C\right) \quad \cC_{\Simp^k} = \cFun\left(\left(\Simpop\right)^{\times k}, \cC\right).
 \end{equation*}

The category of (possibly empty) linear orders $\Delta_+$ has objects labelled by
 \begin{equation*}
  \langle n \rangle = \{1 < \cdots < n\} \in \Delta_+.
 \end{equation*}
From $\Delta_+$ one can build the category $\nabla$ (\cite{Kockres} 8)\footnote{Note that our category $\nabla$ is the opposite of the one defined in \cite{Kockres}}, which has the same objects but morphisms 
\begin{equation}
 \label{nabladefn}
 \xymatrixcolsep{1.1pc} \xymatrixrowsep{.8pc} \xymatrix{ & \ar[ld] \langle k \rangle \ar@{>->}[rd] & \\
 \langle n \rangle &  & \langle m \rangle\, .}
 \end{equation}
 A {\em $k$-fold $\nabla$-object} in an ordinary category $C$ is an object of 
 \begin{equation*}
 C_{\nabla^k} = \Fun \left((\nabla^\op)^{\times k} , C \right).
 \end{equation*}

 Finally, by (\cite{Kockres} 8.2) one has a functor $\fG:\Delta \to \nabla$ which is bijective on objects and full and restricts to isomorphisms $\Deltaac \simeq \Delta_+^\op$ and $\Deltain^{\geq 1} \simeq (\Delta_+)_\mathrm{in}^{\geq 1}$. Furthermore, equipping $\Deltaac$ with the monoidal structure
 \begin{equation*}
 [n] \vee [m] = [n+m]
\end{equation*}
having unit $[0]$ makes the functor $\fG$ a monoidal equivalence, where $\Delta_+$ has the monoidal structure
\begin{equation}
\label{AugMon}
 \langle n \rangle + \langle m \rangle = \langle n+m \rangle
\end{equation}
having unit $\langle 0 \rangle$. Restriction along $\fG$ induces a fully faithful functor
 \begin{equation}
  \label{Augdef}
  \fG^*: C_{\nabla^k} \to C_{\Delta^k}.
 \end{equation}

\section{Lax bialgebras in \texorpdfstring{$\csmtSpan{\cC}$}{the (oo,2)-category of bispans}}
\label{Sec:catbackground}

The universal Hall bialgebra of a double $2$-Segal object $\cX \in \cC_{\Simp^2}$ will be a lax bialgebra object in the symmetric monoidal $\inftwo$-category of bispans in $\cC$, denoted $\csmtSpan{\cC}$. In this section we cover the technical background necessary to define this notion. We begin in Section \ref{Sec:laxfun} with a quick review of symmetric monoidal $\inftwo$-categories and lax functors between them. Then Section \ref{Sec:bispans} covers the symmetric monoidal $\inftwo$-category of bispans. Finally, in Section \ref{Sec:corepbialg} we define lax bialgebra objects.

\subsection{Lax functors between \texorpdfstring{$\inftwo$}{(oo,2)}-categories}
\label{Sec:laxfun}

By an {\em $\inftwo$-category} we mean a simplicial $\infty$-category $\cB \in (\Catoo)_\Simp$ satisfying the Segal conditions,
\begin{equation*}
 \xymatrixcolsep{1.5pc} \xymatrix{ \cB_n \ar[r]^-\sim & \cB_1 \times_{\cB_0} \cdots \times_{\cB_0} \cB_1}, \quad n \geq 2.
\end{equation*}
Moreover, the $\infty$-category $\cB_0$ must be a space, that is, $\cB_0 \in \cS$, and the Segal space
\begin{equation*}
 \xymatrixcolsep{1.5pc} \xymatrix{ \Simpop \ar[r]^-{\cB} & \Catoo \ar[r]^-{(-)^\simeq} & \cS}
\end{equation*}
must be complete \cite{Lurieinftwo}. A functor of $\inftwo$-categories is simply a natural transformation or, equivalently, the $\infty$-category of $\inftwo$-categories, $\Bicatoo$, is a full subcategory of $(\Catoo)_\Simp$.  

A {\em lax functor} $L: \cB \rightsquigarrow \cB'$ between $\inftwo$-categories $\cB, \cB' \in \Bicatoo$ is, informally, a functor in which the $2$-morphisms witnessing the preservation of composition and unitality are not required to be invertible. To formalise this notion, recall that the {\em unstraightening construction} (\cite{HTT} 3.2.0.1) defines an equivalence between functors $\cC \to \Catoo$ and {\em cocartesian fibrations} over $\cC$,
\begin{equation*}
 \Un: \xymatrixcolsep{1.5pc} \xymatrix{ \cFun(\cC,\Catoo) \ar[r]^-\sim & \mathcal{C}\mathrm{ocart}_{/\cC}}.
\end{equation*}

In particular, one can unstraighten an $\inftwo$-category by taking $\cC = \Simpop$. A functor $F:\cB \to \cB'$ unstraightens to a morphism
\begin{equation*}
 \xymatrixrowsep{.8pc}\xymatrixcolsep{.9pc}\xymatrix{\Un(\cB)\ar[rr]^-{\Un(F)} \ar[rd] & & \Un(\cB') \ar[ld] \\
  & \Simpop & }
\end{equation*}
such that $\Un(F)$ preserves cocartesian morphisms. A lax functor $L: \cB \rightsquigarrow \cB'$ is then defined to be a morphism of fibrations
\begin{equation*}
 \xymatrixrowsep{.8pc}\xymatrixcolsep{.9pc}\xymatrix{\Un(\cB)\ar[rr]^-{L} \ar[rd] & & \Un(\cB') \ar[ld] \\
  & \Simpop & }
\end{equation*}
such that $L$ preserves the cocartesian lifts of inert morphisms in $\Simpop$ (\cite{DK12} 9.2.8), that is, cocartesian lifts of morphisms in $\Simpin^\op$.

Using that the $\infty$-category $\Bicatoo$ has finite limits we define, following Lurie (\cite{LurieHA} 2.0.0.7), a {\em symmetric monoidal $\inftwo$-category} to be a functor $\cB^\otimes: \cFinpt \to \Bicatoo$ such that for every $S_\ast \in \cFinpt$,
\begin{equation*}
\xymatrixcolsep{1.5pc} \xymatrix{ \cB^\otimes(S_*) \ar[r]^-\sim & \displaystyle \prod_{s \in S} \cB^\otimes(\{s\}_\ast)}.
\end{equation*}
The $\infty$-category of symmetric monoidal $\inftwo$-categories, $\smBicatoo$, is a full subcategory of the functor category $\cFun(\cFinpt, \Bicatoo)$. 

Finally, one can readily extend the definition of a lax functor to define a {\em symmetric monoidal lax functor}: a functor between $\inftwo$-categories which laxly preserves composition and unitality but still preserves the symmetric monoidal structure.
\begin{defn}
 \label{DefnsmLax}
 A symmetric monoidal lax functor $L: \cB^\otimes \rightsquigarrow (\cB')^{\otimes}$ between symmetric monoidal $\inftwo$-categories $\cB^\otimes, (\cB')^\otimes \in \smBicatoo$ is a morphism of fibrations
 \begin{equation*}
  \xymatrixrowsep{.8pc}\xymatrixcolsep{.9pc}\xymatrix{\Un\left(\cB^\otimes\right)\ar[rr]^-{L} \ar[rd] & & \Un\left((\cB')^\otimes\right) \ar[ld] \\
  & \cFinDel & }
 \end{equation*}
such that $L$ preserves cocartesian lifts of morphisms in $\cFinpt \times \Simpin^\op$.
\end{defn}

\subsection{The twisted arrow construction and the \texorpdfstring{$\inftwo$}{(oo,2)}-category of bispans}
\label{Sec:bispans}

Informally, from an $\infty$-category $\cC$ having finite limits one can construct an $\inftwo$-category having the same objects as $\cC$, $1$-morphisms span diagrams
\begin{equation*}
 \xymatrixrowsep{1.2pc} \xymatrixcolsep{1.5pc} \xymatrix{ & d \ar[rd] \ar[ld] & \\
 c & & c'}
\end{equation*}
$2$-morphisms `spans of spans'
\begin{equation*}
 \xymatrixrowsep{.8pc}\xymatrix{ & d \ar[rd] \ar[ld] & \\
 c & e \ar[u] \ar[d] \ar[r] \ar[l]& c' \\
 & d' \ar[ru] \ar[lu] & }
\end{equation*}
and compositions given by pullback. The cartesian product on $\cC$ endows this $\inftwo$-category with a symmetric monoidal structure. We shall call this the symmetric monoidal $\inftwo$-category of {\em bispans in $\cC$}, and denote it by $\csmtSpan{\cC}$.

Haugseng \cite{RuneSpans} has given a rigorous construction of $\csmtSpan{\cC}$. In a previous work (\cite{penney2017simplicial} 3) we reformulated Haugseng's construction in a form suitable for our purposes. We shall only quickly review what is needed for the current paper and refer the reader to the previous paper for further details.

The morphisms in this $\inftwo$-category are given in terms of the {\em twisted arrow construction}: a functor
\begin{equation*}
 \Pyr{-}: \Cat \to \Cat,
\end{equation*}
where $\Pyr D$ is the category whose objects are arrows $f:d \to d'$ and whose morphisms from $f_1$ to $f_2$ are diagrams
\begin{equation*}
 \xymatrixrowsep{.9pc} \xymatrix{d_1 \ar[r]^-{f_1} \ar[d] & d_1' \\
 d_2 \ar[r]_-{f_2} & d_2' \ar[u] }
\end{equation*}
The $\infty$-categories $\cPyr{D}$ are the nerves of the categories $\Pyr {D}$.

\begin{remark}
\label{Rem:oplax}
A fact which will we use repeatedly in the latter parts of this paper is the following: a functor $\Pyr D \to C$, for $C$ a category having finite limits, is equivalent to a {\em normal oplax functor} $D \nrightarrow \sp C$ (\cite{Errington} 3.4.1). The target $\sp C$ is a bicategory having the same objects as $C$, $1$-morphisms spans and $2$-morphisms diagrams 
\begin{equation*}
 \xymatrixrowsep{.7pc} \xymatrix{ & d \ar[rd] \ar[ld] \ar[dd] & \\
 c & & c' \, .\\
 & d' \ar[ru] \ar[lu] & }
\end{equation*}
\end{remark}

For a poset $X$, the twisted arrow category $\Pyr X$ is the opposite of the poset of non-empty subintervals in $X$. In particular, for any $\phi \in N_k(\Delta^\op)$, where $N_\bullet(-)$ is the nerve, one has a poset $M_\phi$ with objects
\begin{equation*}
 \{ (a,b) \ | \ b \in[k]^{\rm op}, a \in \phi(b)\},
\end{equation*}
and ordering defined by declaring $(a,b) \leq (a',b')$ if and only if $b \leq b' \in [k]^{\rm op}$ and $\phi_{b,b'}(a) \leq a'$. 
\begin{exam}
For $\phi \in N_k(\Delta^{\rm op})$ the constant map on $[n]$, the poset $M_\phi$ is isomorphic to $[n]\times[k]$.
\end{exam}
\begin{exam}
 For the unique active morphism $\phi = ([2] \actmorL [1]) \in N_1(\Delta^{\rm op})$, the poset $M_\phi$ is
\begin{equation*}
 \xymatrixrowsep{.8pc}\xymatrixcolsep{1.1pc} \xymatrix{(0,1) \ar[d] \ar[rr] & & (1,1) \ar[d] \\
 (0,0) \ar[r] & (1,0) \ar[r] & (2,0)}
\end{equation*}
\end{exam}

Since the posets $M_\phi$ are obtained from the Grothendieck construction of the functor
\begin{equation*}
 \xymatrix{ [k]^{\rm op} \ar[r]^-\phi & \Delta \ar@{^{(}->}[r] & \Cat},
\end{equation*}
they satisfy the following naturality properties: each natural transformation $\eta: \phi' \Rightarrow \phi$ gives a functor 
\begin{equation*}
M(\eta): M_{\phi'} \to M_{\phi}, \ (a,b) \mapsto (\eta_b(a), b)
\end{equation*}
and each morphism $\gamma: [n] \to [k]$ gives a functor 
\begin{equation*}
M(\gamma): M_{\phi \gamma} \to M_{\phi}, \ (a,b) \mapsto (a, \gamma(b)).
\end{equation*}

There are two relevant subcategories of $\Pyr{M_\phi}$. The first is the subcategory $\Wedge{M_\phi}$ consisting of those intervals 
\begin{equation*}
 \Wedge{M_\phi} = \left\{ [(a,b);(\phi_{b,b'}(a'),b')] \ | \ |b'-b| \leq 1 \, \mathrm{and} \, |a'-a| \leq 1 \right\}.
\end{equation*}
 The second is the subcategory $\Pyr{V_\phi}$, where $V_\phi$ is the subcategory of $M_\phi$ on those intervals
 \begin{equation*}
  V_\phi = \left\{ [(a,b);(\phi_{b,b'}(a),b')]\right\}.
 \end{equation*} 
\begin{exam}
 For $\phi \in N_k(\Delta^\op)$ the constant map on $[n]$, the subcategory $\Wedge{M_\phi}$ is isomorphic to $\Wedge{n,k} := \Wedge{n} \times \Wedge{k}$, where $\Wedge {n}$ is the subcategory of $\Pyr n$ on those intervals $[i;j]$ with $|j-i|\leq 1$. The category $V_\phi$ is isomorphic to $V_{n,k}$, the subcategory of $[n]\times[k]$ on the morphisms $(\id, g)$.
\end{exam}
\begin{exam}
  \label{Vphiact}
  For the unique active morphism $\phi = ([2] \actmorL [1]) \in N_1(\Delta^{\rm op})$, the subcategory $\Wedge {M_\phi}$ is
\begin{equation*}
 \xymatrixcolsep{.75pc}\xymatrixrowsep{.8pc} \xymatrix{[(0,1);(0,1)] & & \ar[ll] [(0,1);(1,1)] \ar[rr] & & [(1,1);(1,1)] \\
 [(0,1);(0,0)] \ar[u] \ar[d] & & [(0,1);(2,0)] \ar[ll] \ar[u] \ar[rr] \ar[rd] \ar[ld] & & [(1,1);(2,0)] \ar[u] \ar[d] \\
 [(0,0);(0,0)] & \ar[l] [(0,0);(1,0)] \ar[r] & [(1,0);(1,0)] & \ar[l] [(1,0);(2,0)] \ar[r] & [(2,0);(2,0)] }
\end{equation*}
  and the poset $V_\phi$ is
\begin{equation*}
 \xymatrixrowsep{.8pc}\xymatrixcolsep{1.1pc} \xymatrix{(0,1) \ar[d]  & & (1,1) \ar[d] \\
 (0,0)  & (1,0)  & (2,0)}
\end{equation*}
\end{exam}

Finally, the categories $\Cart S$, the poset of subsets of a set $S$, assemble into a functor $\Cart {(-) }: \Finpt^{\rm op} \to \Cat$ by declaring the image of a pointed map $f:S_* \to T_*$ to be
 \begin{equation*}
  \Cart f: \Cart T \to \Cart S, \quad U \mapsto f^{-1}(U).
 \end{equation*}
Each category $\Cart S$ has a full subcategory $\Sing S$ consisting of the singleton sets. The $\infty$-categories $\cSing S$ and $\cCart S$ are, respectively, the nerves of $\Sing S$ and $\Cart S$.

\begin{defn}
 Let $\cC$ be an  $\infty$-category with finite limits, $S$ a set and $\phi \in N_k(\Delta^\op)$. Then we say a functor $\tau: \cCartop{S}\times \cPyr {M_\phi} \to \cC$ is:
 \begin{enumerate}
  \item {\em cartesian} if it is the right Kan extension of its restriction to $\cSingop{S} \times \cWedge{M_\phi}$.
  \item {\em vertically constant} if morphisms of the form $(\id, v)$ for $v \in \cPyr{V_\phi}$ are sent to isomorphisms.
 \end{enumerate}
\end{defn}

\begin{exam}
\label{Sigmaintuit}
 A functor $\cPyr 2 \to \cC$ is a diagram,
\begin{equation*}
  \xymatrixrowsep{.8pc}\xymatrix{ & & c_{[0;2]} \ar[rd] \ar[ld] & & \\
  & c_{[0;1]} \ar[rd] \ar[ld] & & c_{[1;2]} \ar[rd] \ar[ld] & \\
  c_{[0;0]} & & c_{[1;1]} & & c_{[2;2]} \, . }
 \end{equation*}
Such a diagram is cartesian if it is the right Kan extension of its restriction to $\cWedge 2$, i.e., if the middle square is a pullback in $\cC$. In general, a functor $\cPyr n \to \cC$ is pyramid of spans on $n+1$ objects. It being cartesian says higher tiers of this pyramid consist of a coherent choice of pullbacks of the $n$ spans along the bottom two tiers. 
 \end{exam}
  
 \begin{exam}
A functor $\cCartop {\{1,2\}}\to \cC$ is a diagram,
\begin{equation*}
 \xymatrixrowsep{.8pc}\xymatrix{ & c_{\{1,2\}} \ar[rd] \ar[ld] & \\
 c_{\{1\}} \ar[rd] & &\ar[ld] c_{\{2\}} \\
  & c_{\emptyset} & }
\end{equation*}
Such a diagram is cartesian if it presents $c_{\{1,2\}}$ as the product of $c_{\{1\}}$ and $c_{\{2\}}$ and $c_{\emptyset}$ is terminal. Similarly, a cartesian functor $\cCartop S\to \cC$ encodes a coherent choice of products for a collection of objects of $\cC$ labelled by the elements of $S$.
\end{exam}

\begin{defn}[\cite{penney2017simplicial} 3.11]
Let $\cC$ be an $\infty$-category having finite limits. Then the {\em symmetric monoidal $\inftwo$-category of bispans in $\cC$}, denoted $\csmtSpan \cC$, is given by the functor
\begin{equation*}
 \FinDel \to \qCat, \quad (S_*, [n]) \mapsto \Sp_{S,n} (\cC),
\end{equation*}
where $\Sp_{S,n}(\cC)$ is the quasi-category having $k$-simplices
\begin{equation*}
 \left\{ \cCartop S \times \cPyr {n,k} \to \cC \ | \ {\rm cartesian}, \ {\rm vertically} \  {\rm constant}\right\}
\end{equation*} 
\end{defn}

In Section \ref{Sec:totlax} we will be writing explicit functors into the unstraightening of $\csmtSpan{\cC}$ in the special case of $\cC$ the nerve of an ordinary category $C$ having finite limits. To that end, we will make use of the following explicit description of the unstraightening:
\begin{prop}[\cite{penney2017simplicial} 3.24]
\label{UnstrSimps}
Let $C$ be a category with finite limits, and let $\Unstr C_k$ be the set
\begin{equation*}
 \left\{ \left((f,\phi)\in N_k(\FinDel), \, \tau: \Cartop{f(0)}\times\Pyr{M_\phi} \to C\right) \ | \ \tau \ {\rm cartesian}, \ {\rm vertically} \  {\rm constant}\right\}.
\end{equation*}
Then the sets $\Unstr C_k$ assemble into a sub simplicial set of the unstraightening of $\csmtSpan {N(C)}$.
\end{prop}

\subsection{Corepresenting bialgebra objects}
\label{Sec:corepbialg}

A bialgebra object in a symmetric monoidal $\inftwo$-category is, just like in an ordinary symmetric monoidal category, an object equipped with compatible algebra and coalgebra structures. The compatibility condition requires that the data defining the coalgebra structure be algebra homomorphisms, or equivalently, that the data defining the algebra structure be coalgebra homomorphisms. The data required for defining a bialgebra object can be corepresented by a combinatorially defined symmetric monoidal $\infty$-category $\csmBialg$ originally introduced by Pirashvili \cite{Pirash02} and Lack \cite{Lack}.

Naively, one would like to define $\cBialg$ to be $\cSpan{\Alg}$, where $\Alg$ is the symmetric monoidal category which corepresents algebra objects (\cite{penney2017simplicial} 2.3): the category $\Alg$ has objects finite sets and morphisms functions $p:X \to Y$ equipped with a linear ordering on $p^{-1}(y)$ for each $y \in Y$. Disjoint union endows $\Alg$ with a symmetric monoidal structure. Unfortunately $\Alg$ does not admit all pullbacks. Temporarily putting that issue aside, let us consider how $\cSpan{\Alg}$ could corepresent bialgebra objects.

 By considering morphisms of the form
\begin{equation*}
 \xymatrixrowsep{.2pc} \xymatrix{ & X \ar[rdd] \ar@{=}[ldd] &  &  & & X \ar[ldd] \ar@{=}[rdd] & \\
  & & & {\rm and} & & & \\
  X & & Y & & Y & & X }
\end{equation*}
one observes that $\cSpan{\Alg}$ contains $\Alg$ and $\Coalg:=\Alg^\op$ as wide subcategories. Therefore a symmetric monoidal functor out of $\cSpan{\Alg}$ specifies an object carrying both an algebra and a coalgebra structure. The compatibility of these structures arises from the requirement that functors preserve composition of spans. For example, the condition that the coproduct respects products arises from the composite
\begin{equation*}
 \xymatrixrowsep{.8pc} \xymatrix{ & & \underline{4} \ar[ld] \ar[rd] & & \\
 & \underline{2} \ar@{=}[ld] \ar[rd] & & \underline{2} \ar@{=}[rd] \ar[ld] & \\
 \underline{2} & & \underline{1} & & \underline{2} \ .}
\end{equation*}

To give a proper definition of $\cBialg$ we must find a way to compose span diagrams in $\Alg$ despite its lack of pullbacks. 

\begin{defn}[\cite{Pirash02}]
A (not necessarily commutative) square in $\Alg$
\begin{equation*}
 \xymatrixrowsep{1.1pc}\xymatrix{X' \ar[r]^{p'} \ar[d]_{q'} & Y' \ar[d]^q \\
 X \ar[r]_{p} & Y}
\end{equation*}
is called a {\em pseudo-pullback square} if
\begin{enumerate}
 \item the image of the square under the forgetful functor $\Alg \to \Fin$ is a pullback square;
 \item for each $x \in X$, the induced map $p_*: (q')^{-1}(x) \to q^{-1}(p(x))$ is an isomorphism of ordered sets; and,
 \item for each $y \in Y'$ the induced map $q_*: (p')^{-1}(y) \to p^{-1}(q(y))$ is an isomorphism of ordered sets.
\end{enumerate}
\end{defn}

\begin{exam}
 Consider the following pullback square of finite sets
 \begin{equation*}
  \xymatrixrowsep{1.1pc} \xymatrix{\underline{4} \ar[d]_-{m^2} \ar[r]^-{\overline{m}^2} & \underline{2} \ar[d]^-{m} \\
  \underline{2} \ar[r]_-{m} & \underline{1} }
 \end{equation*}
Denote by $\underline{2} = \{y_1, y_2\}$, $\underline{4} = \{y_{1,1}, y_{1,2}, y_{2,1}, y_{2,2}\}$, $m^2(y_{i,j}) = y_i$ and $\overline{m}^2(y_{i,j}) = y_j$. If one lifts this to a diagram in $\Alg$ by setting the ordering on the fibres of $m$ to be $y_1 \leq y_2$, such a square is a pseudo-pullback when one equips $m^2$ and $\overline{m}^2$ with the orderings
\begin{equation*}
 (m^2)^{-1}(y_i) = (y_{i,1} \leq y_{i,2}) \ {\rm and} \ (\overline{m}^2)^{-1}(y_i) = (y_{1,i} \leq y_{2,i}).
\end{equation*}
The resulting diagram in $\Alg$ {\em does not commute}: the composites $m \circ m^2$ and $m \circ \overline{m}^2$ define, respectively, the following inequivalent orderings on $\underline{4}$:
\begin{equation*}
 (y_{1,1}\leq y_{1,2} \leq y_{2,1} \leq y_{2,2}) \ {\rm and} (y_{1,1}\leq y_{2,1} \leq y_{1,2} \leq y_{2,2}).
\end{equation*}
\end{exam}

The failure of commutativity in the above example is a general property of pseudo-pullback squares in $\Alg$. Nonetheless, since pseudo-pullback squares in $\Alg$ are sent to pullback squares in $\Fin$ under the forgetful functor and isomorphisms in $\Fin$ have unique lifts in $\Alg$, such squares do satisfy all of the properties required of pullback squares to define the composition in a category of spans. The only care that must be taken is that one can no longer use diagrams of the form $\Pyr n$ to define strings of composable morphisms.

For each $n\geq 0$, let $G^n$ be the reflexive directed graph having vertices subintervals $[a;b] \subset [n]$. There is an edge $[a;b] \to [a'; b']$ if and only if $[a';b'] \subset [a;b]$ and $[a;b]$ contains at most one more element than $[a';b']$. Define the categories $\Pyrnc n$ to be the free categories on $G^n$.

\begin{exam}
 A functor $\Pyrnc 2 \to \Alg$ is a diagram,
 \begin{equation*}
  \xymatrixrowsep{.8pc}\xymatrix{ & & X_{[0;2]} \ar[rd] \ar[ld] \ar@<.7ex>[dd] \ar@<-.7ex>[dd] & & \\
   & X_{[0;1]} \ar[rd] \ar[ld] & & X_{[1;2]} \ar[rd] \ar[ld] & \\
   X_{[0;0]} & & X_{[1;1]} & & X_{[2;2]} }
 \end{equation*}
A $\Pyrnc n$-diagram is a not-necessarily-commutative diagram in the same shape as a $\Pyr n$-diagram.
\end{exam}

The categories $\Pyrnc n$ assemble into a functor $\Pyrnc{\bullet}:\Delta \to \Cat$. A morphism $\phi:[n] \to [m]$ is sent to the functor $\Pyrnc{\phi}:\Pyrnc n \to \Pyrnc m$ which is defined on objects by sending $[a;b]$ to $[\phi(a);\phi(b)]$ and on the generating morphisms by
\begin{eqnarray}
 \left([a;b] \to [a+1;b]\right) &\mapsto& \left(\xymatrixcolsep{.6pc}\xymatrix{[\phi(a);\phi(b)]\ar[r]& [\phi(a)+1;\phi(b)] \ar[r] & \cdots \ar[r] & [\phi(a+1);\phi(b)]}\right) \\
 \left([a;b] \to [a;b-1]\right) &\mapsto& \left(\xymatrixcolsep{.6pc}\xymatrix{[\phi(a);\phi(b)]\ar[r]& [\phi(a);\phi(b)-1] \ar[r] & \cdots \ar[r] & [\phi(a+1);\phi(b-1)]}\right) \ . \nonumber
\end{eqnarray}

We define a functor $F: \Pyrnc n \to \Alg$ to be {\em pseudo-cartesian} if for each $\phi:[2] \to [n]$ the middle square of $F\Pyrnc{\phi}: \Pyrnc 2 \to \Alg$ is a pseudo-pullback.

\begin{defn}
\label{BialgDefn}
 For each $(S_*,[n]) \in \FinDel$, let $\cBialg_{S,n}$ denote the nerve of the groupoid $\Bialg_{S,n}$ of functors $F:\Cart S \times \Pyrnc n \to \Alg$ such that
\begin{enumerate}
 \item For each $U \in \Cart S$, the functor $F(U,-):\Pyrnc n \to \Alg$ is pseudo-cartesian.
 \item For each $x \in \Pyrnc n$, the functor $F(-,x):\Cart S \to \Alg$ is cocartesian.
\end{enumerate}
These assemble into a symmetric monoidal $\inftwo$-category
\begin{equation*}
 \csmBialg: \FinDel \to \qCat, \quad (S_*,[n]) \mapsto \cBialg_{S,n}.
\end{equation*}
\end{defn}

Note that projecting on to the right and left endpoints define, respectively, natural transformations $\Pyrnc{\bullet} \Rightarrow [\bullet]$ and $\Pyrnc{\bullet} \Rightarrow [\bullet]^{\rm op}$. These induce, respectively, symmetric monoidal functors
\begin{equation*}
 \xymatrixcolsep{1.2pc}\xymatrix{\csmAlg \ar[r]^-{\iota_a}& \csmBialg & {\rm and} & \csmCoalg \ar[r]^-{\iota_c} & \csmBialg\, ,}
\end{equation*}
where $\csmAlg$ and $\csmCoalg$ are the nerves of $\Alg$ and $\Coalg$. 

The following bialgebraic structures will be the focus of this paper.
\begin{defn}
\label{DefnLaxBialg}
Let $\cB^\otimes$ be a symmetric monoidal $\inftwo$-category.
\begin{itemize}
 \item A {\em bialgebra object} in $\cB^\otimes$ is a symmetric monoidal functor $\csmBialg \to \cB^\otimes$.
 \item A {\em totally lax bialgebra object} in $\cB^\otimes$ is a symmetric monoidal lax functor $\csmBialg \rightsquigarrow \cB^\otimes$.
 \item A {\em lax bialgebra object} in $\cB^\otimes$ is a totally lax bialgebra object in $\cB^\otimes$ which restricts to symmetric monoidal functors $\csmAlg \to \cB^\otimes$ and $\csmCoalg \to \cB^\otimes$, i.e.,
 \begin{equation*}
  \xymatrixrowsep{.9pc} \xymatrix{ \csmAlg\ar[d]_-{\iota_a} \ar[rd] & \\
  \csmBialg \ar@{~>}[r] & \cB^\otimes \\
  \csmCoalg\ar[u]^-{\iota_c} \ar[ru] & }
 \end{equation*}
\end{itemize}
\end{defn}

Informally, a totally lax bialgebra is an object equipped with both a laxly associative product and laxly coassociative coproduct which are laxly compatible. A lax bialgebra is a totally lax bialgebra for which the product is associatve and coproduct is coassociative.

\section{Double \texorpdfstring{$2$}{2}-Segal spaces}
\label{Sec:double2Seg}

Taking for granted the definition of a $2$-Segal space, which we review in Section \ref{Sec:Rev2Seg}, the notion of a $2$-Segal space is straightforward to define: It is a bisimplicial space $\cX_{\bullet,\bullet}$ such that the simplicial spaces $\cX_{\bullet,k}$ and $\cX_{k,\bullet}$ are $2$-Segal for each $k$. The purpose of this section is to populate the world of double $2$-Segal spaces by constructing examples using the $S$-construction of augmented proto-exact $\infty$-categories. This is done in Section \ref{Sec:augexact}.

\subsection{Review of \texorpdfstring{$2$}{2}-Segal spaces}
\label{Sec:Rev2Seg}

Dyckerhoff--Kapranov (\cite{DK12} 2.3.2) were the first to introduce a $2$-dimensional generalisation of the usual Segal condition for a simplicial space: the {\em $2$-Segal condition}. Independently, G\'alvez-Carrillo--Kock--Tonks (\cite{KockI} 3) provided an equivalent definition which requires certain pushout squares in $\Simp$ to be sent to pullbacks. In a previous work \cite{penney2017simplicial} we have given a third formulation of the $2$-Segal condition which shall be the most convenient for our purposes. 

Let $\cX \in \cC_\Simp$ be a simplicial object in an $\infty$-category $\cC$ having finite limits. We have shown (\cite{penney2017simplicial} 4.1) that the object of $1$-simplices $\cX_1$ canonically carries the structure of a {\em lax algebra object} in $\csmtSpan{\cC}$: One has the functor
\begin{equation*}
  \alpha_\cX: \xymatrix{\csmAlg \ar@{~>}[r]^-{\alpha} & \csmtSpan{\csSetfop} \ar[r] & \csmtSpan{\cC}},
 \end{equation*}
where the second functor is obtained from the right Kan extension
\begin{equation*}
 \xymatrixrowsep{1.1pc} \xymatrix{ \csSetfop \ar[r] & \cC \\
 \Simpop \ar@{^{(}->}[u] \ar[ru]_-\cX & }
\end{equation*}
and $\alpha$ is a symmetric monoidal lax functor endowing the standard $1$-simplex $\Simplex 1$ with the structure of lax algebra.

Informally, the lax functor $\alpha$ is, on objects, simply 
\begin{equation*}
\alpha:X\mapsto X\cdot \Simplex 1 = \coprod_{x \in X} \Simplex 1.
\end{equation*}
The lax functor $\alpha$ sends the morphism $p$ to the morphism $\alpha(p) \in \csmtSpan{\csSetfop}$ given by the diagram
\begin{equation*}
 \xymatrixrowsep{.8pc}\xymatrix{ & \displaystyle \coprod_{y \in Y} \Simplex{|p^{-1}(y)|} & \\
 X\cdot \Simplex 1 \ar[ru]^-{\sigma} & & \ar[lu]_-{\lambda} Y\cdot \Simplex 1 \, ,}
\end{equation*}
The morphism $\lambda$ sends the $1$-simplex associated to the element $y \in Y$ to the long edge of the standard simplex $\Simplex{|p^{-1}(y)|}$. Note that one can label the edges along the spine of $\Simplex{|p^{-1}(y)|}$ by the elements of $p^{-1}(y)$ using the linear ordering. The morphism $\sigma$ sends the $1$-simplex associated to $x \in X$ to the appropriate edge along the spine of $\Simplex{|p^{-1}(p(x))|}$. 

\begin{exam}
For the morphism $m:\underline{2} \to \underline{1}$, the morphism $\alpha(m)$ in $\csmtSpan{\csSetfop}$ is given by the diagram 
\begin{equation*}
 \includegraphics[height=1.45cm]{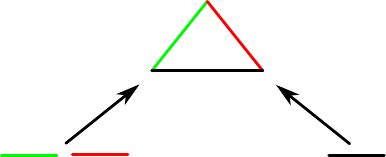}
\end{equation*}
\end{exam}
\begin{exam}
  Consider the morphism $(m \amalg \id):\underline{3} = \underline{2} \amalg \underline{1} \to \underline{1} \amalg \underline{1} = \underline{2}$. Then $\alpha(m \amalg \id)$ is given by the diagram
  \begin{equation*}
   \includegraphics[height=1.45cm]{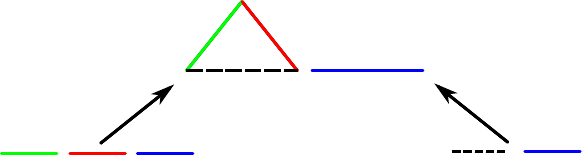} 
  \end{equation*}
\end{exam}

The lax structure on the functor $\alpha$ is given by associating to each pair of composable morphisms $\xymatrixcolsep{1.1pc}\xymatrix{X_0 \ar[r]^-{p_1} & X_1 \ar[r]^-{p_2} & X_2}$ in $\csmAlg$ a $2$-morphism in $\csmtSpan{\csSetfop}$ of the form
\begin{equation*}
 \xymatrixrowsep{.8pc}\xymatrix{ & \alpha(p_2p_1) \ar@{=}[d] & \\
 \alpha(X_0) \ar[ru] \ar[rd] \ar[r] & \alpha(p_2p_1) & \alpha(X_2) \ . \ar[lu] \ar[ld] \ar[l] & \\
  & \alpha(p_2) \coprod_{\alpha(X_1)} \alpha(p_1) \ar[u] & }
\end{equation*}

\begin{exam}
Consider the pair of composable morphisms $\xymatrixcolsep{1.7pc}\xymatrix{\underline{3} \ar[r]^-{m \amalg \id} & \underline{2} \ar[r]^-{m} & \underline{1}}$. The lax structure on $\alpha$ is given by the diagram
\begin{equation*}
 \includegraphics[height=4.5cm]{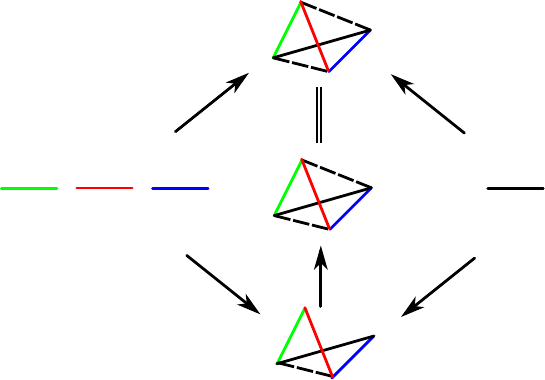} 
\end{equation*}
\end{exam}

\begin{defn}[\cite{penney2017simplicial} 4.14]
 \label{Defn:2Seg}
 A simplicial object $\cX \in \cC_\Simp$ is a $2$-Segal object if and only if $\alpha_\cX$ is a symmetric monoidal functor.
\end{defn}
\begin{remark}
 What we call a $2$-Segal object is called a {\em unital $2$-Segal object} in \cite{DK12} and a {\em decomposition space} in \cite{KockI}.
\end{remark}
For a $2$-Segal object $\cX \in \cC_\Simp$, the functor $\alpha_\cX$ makes $\cX_1$ an algebra object in $\csmtSpan{\cC}$ called the {\em universal Hall algebra} of $\cX$. It is so named since Dyckerhoff--Kapranov have shown (\cite{DK12} 8) that various Hall algebra-like constructions that have appeared in the literature can be seen as particular linearisations of the universal Hall algebra.

The $\infty$-category of $2$-Segal objects in $\cC$, denoted by $\ctSeg{\cC}$, is defined to be the full subcategory of $\cC_\Simp$ on the $2$-Segal objects. The $2$-Segal condition can also be formulated as requiring that $\cX$ send so-called {\em active-inert pushout squares} in $\Simp$ to pullback squares in $\cC$ (\cite{KockI} 3.1). This formulation implies that $\ctSeg \cC$ has finite limits which are computed object-wise in $\cC$. With this in hand one can give a concise definition of a double $2$-Segal object:
\begin{defn}
 \label{Defn:double2Seg}
 A double $2$-Segal object in an $\infty$-category $\cC$ is a $2$-Segal object $\cX$ in $\ctSeg \cC$, the $\infty$-category of $2$-Segal objects in $\cC$. 
\end{defn}

Next, in Section \ref{Sec:augexact} we discuss how to construct examples of $2$-Segal and double $2$-Segal objects.

\subsection{The \texorpdfstring{$S$}{S}-construction of augmented proto-exact \texorpdfstring{$\infty$}{oo}-categories}
\label{Sec:augexact}
Our main examples of double $2$-Segal spaces will arise from the Waldhausen $S$-construction \cite{Waldhausen} of {\em augmented proto-exact $\infty$-categories}: not necessarily additive, or even pointed, generalisations of Barwick's {\em exact $\infty$-categories} \cite{barwick2015exact}. We will begin by developing their basic theory as these technical developments make the construction of double $2$-Segal spaces straightforward.

To give the definition of an augmented proto-exact $\infty$-category we must first recall some terminology. First, a morphism $f:c \to d$ in an $\infty$-category $\cC$ is {\em monic} if the commutative square
\begin{equation*}
 \xymatrixrowsep{1.1pc} \xymatrix{ c \ar@{=}[r] \ar@{=}[d] & c \ar[d]^-f \\
 c \ar[r]_-f & d}
\end{equation*}
is a pullback square. In particular, a functor $F:\cA \to \cB$ between $\infty$-categories is a monic morphism in $\Catoo$ if and only if it is an inclusion of subcategories which is full on equivalences. Next, a functor $F: \cC \to \cD$ of $\infty$-categories is {\em final} if precomposition with $F$ preserves colimits, or equivalently, if for each $c \in \cC$ the space $(F^{c/})^\mathrm{gpd}$ is contractible. Dually, $F$ is {\em initial} $F^\op$ is final.

\begin{defn}
 \label{Def:augexact}
 An augmented proto-exact $\infty$-category is an $\infty$-category $\cA$ equipped with subcategories of {\em null objects} $\cN$, {\em admissible monomorphisms} $\cM$ and {\em admissible epimorphisms} $\cE$ that fit into a diagram
\begin{equation}
\label{APEsubs}
\xymatrixrowsep{1.1pc} \xymatrix{ \cN \ar[r]^-{0_\cE} \ar[d]_-{0_\cM} & \cE \ar[d]^-{\iota_\cE} \\
\cM \ar[r]_-{\iota_\cM} & \cA}
\end{equation}
such that:
\begin{enumerate}
 \item \label{Cond:monic} the functors $0_\cM$, $0_\cE$, $\iota_\cE$ and $\iota_\cM$ are monic; 
 \item \label{Cond:null} the subcategory $\cN$ is in $\cS$, that is, has only invertible morphisms; 
 \item \label{Cond:init}for any $a \in \cA$, the spaces $(0_\cM)_{/a}$ and $0_\cE^{a/}$ are contractible and the functors
 \begin{equation*}
  \xymatrixcolsep{1.1pc} \xymatrix{ (0_\cM)_{/a} \ar[r] & \cA_{/a} &  0_\cE^{a/} \ar[r] & \cA^{a/}}
 \end{equation*}
are, respectively, final and initial.
 \item \label{Cond:bicart} any commutative diagram in $\cA$
 \begin{equation}
 \label{APEsquare}
  \xymatrixrowsep{1.1pc} \xymatrix{ A_1 \ar[d]_-{f} \ar[r]^-{g_1} & A_2 \ar[d]^-{f'} \\
  A_1' \ar[r]_-{g_2} & A_2'} 
 \end{equation}
   with $f,f' \in \cE$ and $g_1,g_2 \in \cM$ is a pullback if and only if it is a pushout;
 \item \label{Cond:monostab}morphisms in $\cM$ admit and are stable under pushouts along morphisms in $\cE$; and 
 \item \label{Cond:epistab} morphisms in $\cE$ admit and are stable under pullbacks along morphisms in $\cM$.
\end{enumerate}
\end{defn}
\begin{remark}
 Condition \ref{Cond:init} above says that for each $a \in \cA$ there is an essentially unique morphism $z \to a \in \cM$ and $a \to z' \in \cE$ with $z$ and $z'$ being possibly distinct objects of $\cN$. As such, every augmented proto-exact $\infty$-category $\cA$ which is the nerve of an ordinary category defines an augmented stable double category in the sense of Bergner--Osorno--Ozornova--Rovelli--Scheimbauer \cite{WiTWald}. A work in progress from the same authors introduces the notion of an augmented stable double Segal space, and we suspect that every augmented proto-exact $\infty$-category defines such an object.
\end{remark}

The main class of examples of augmented proto-exact $\infty$-categories are familiar inputs for the $\infty$-categorical Waldhausen $S$-construction.
\begin{exam}
An augmented proto-exact $\infty$-category $\cA$ which is pointed and whose null objects $\cN$ is the full subcategory of zero objects is a proto-exact $\infty$-category (\cite{DK12} 7.2.1)\footnote{Note that what we call proto-exact $\infty$-categories Dyckerhoff--Kapranov merely call exact $\infty$-categories. Our terminology is consistent with the notion of a proto-exact category introduced by these authors in Section 2.4 of the same paper and serves to distinguish such categories from Barwick's exact $\infty$-categories}. If $\cA$ is in addition additive, then it is an exact $\infty$-category (\cite{barwick2015exact} 3.1). This class of examples includes familiar examples such as stable $\infty$-categories and the nerves of Quillen exact categories.
\end{exam}

In the above examples the spaces of null objects was contractible. The following example serves as the main motivation for expanding the definition to allow for many connected components.
\begin{exam}
 For a poset $X$, let $\cArr X$ denote the $\infty$-category of functors $\cFun([1],N(X))$. Objects of $\cArr X$ correspond to pairs $ij$ with $ i \leq j$ in $X$ and one has a morphism $ij \to i'j'$ if and only if $i \leq i'$ and $j \leq j'$ The $\infty$-category $\cArr X$ can be made into an augmented proto-exact $\infty$-category as follows: Set $\cM$ to consist of those morphisms $ij \to ij'$, set $\cE$ to be those morphisms and $ij \to i'j$ and $\cN$ to be the full subcategory on the objects $ii$. 
\end{exam}

A functor between augmented proto-exact $\infty$-categories is said to be {\em exact} if it preserves the subcategories of null objects and admissible mono- and epi- morphisms, and sends bicartesian squares of the form Eq.~\ref{APEsquare} to bicartesian squares. For any two augmented proto-exact $\infty$-categories $\cA$ and $\cA'$ one has the full subcategory 
\begin{equation*}
\xymatrixcolsep{1.1pc} \xymatrix{\cFun^\mathrm{ex}(\cA,\cA')\ar@{^{(}->}[r] & \cFun(\cA,\cA')}
\end{equation*}
on the exact functors.

\begin{defn}
Let $\apExact$ denote the simplicially enriched category whose objects are the augmented proto-exact $\infty$-categories and whose mapping spaces are $\cMap^\mathrm{ex}(-,-)$. The {\em $\infty$-category of augmented proto-exact $\infty$-categories}, $\capExact$, is the simplicial nerve of $\apExact$.
\end{defn}

Before moving on with our study of the categorical properties of $\capExact$, let us consider a few examples of exact functors between augmented proto-exact $\infty$-categories.
\begin{exam}
 An exact functor $F:\cC \to \cC'$ between exact $\infty$-categories (\cite{barwick2015exact} 4.1) is also an exact functor between augmented proto-exact $\infty$-categories.
\end{exam}

\begin{exam}
\label{arrcosimp}
 Post-composing with an order-preserving morphism $\phi: X \to Y$ yields a functor $\cArr \phi: \cArr X \to \cArr Y$ which is exact. Restricting to the non-empty linearly ordered posets defines a cosimplicial object in $\capExact$,
 \begin{equation*}
   \cArr \bullet: \Simp \to \capExact.
 \end{equation*}
\end{exam}

According to Definition \ref{Def:augexact}, one has a faithful functor
\begin{equation*}
 \xymatrix{ \capExact \ar@{^{(}->}[r] & \Catoo^\Box:= \cFun\left([1]\times[1],\Catoo\right)}
\end{equation*}
exhibiting $\capExact$ as a subcategory of $\Catoo^\Box$ by sending an augmented proto-exact $\infty$-category $\cA$ to the diagram in Eq.~\ref{APEsquare}. 
\begin{lem}
 \label{Lem:finlim}
 The $\infty$-category of augmented proto-exact $\infty$-categories, $\capExact$, is complete and the functor $\capExact \to \Catoo^\Box$ preserves limits.
\end{lem}
\begin{proof}
 It suffices to show that for any diagram $D: \cJ \to \capExact$ the limit of the functor
 \begin{equation*}
  \xymatrix{ \cJ \ar[r]^-D & \capExact \ar@{^{(}->}[r] & \Catoo^\Box}
 \end{equation*}
is an augmented proto-exact $\infty$-category. To see this one need only observe that since limits in $\Catoo^\Box$ are computed pointwise, a functor $\cA \to \lim_{\cJ} D$ for $\cA \in \capExact$ is exact if and only if each of the component functors $\cA \to D(j)$ is exact.

Let $\cD$ denote the limit $\infty$-category $\lim_{\cJ} D$ equipped with subcategories $\cM = \lim_{\cJ} \cM(j)$, $\cE = \lim_{\cJ} \cE(j)$ and $\cN = \lim_{\cJ} \cN(j)$. Then $\cD$ satisfies Condition \ref{Cond:monic} of Definition \ref{Def:augexact} since monic morphisms are preserved by limits. Similarly, it satisfies Condition \ref{Cond:null} since $\cS$ is a full subcategory of $\Catoo$. 

As for Condition \ref{Cond:init}, for $d$ an object of $\cD$ let $d(j)$ be its image in $D(j)$. Then $(0_\cM)_{/d}$ and $0_\cE^{d/}$ are contractible since 
\begin{equation*}
 (0_\cM)_{/d} \simeq \lim_\cJ \left(0_{\cM(j)}\right)_{/d(j)} \ \mathrm{and} \ 0_\cE^{d/} \simeq \lim_\cJ 0_{\cE(j)}^{d(j)/}.
\end{equation*}
For any pair of objects $d'$ and $d$ in $\cD$ one has that $\cF(d',d) \simeq \lim_\cJ \cF(d'(j),d(j))$ where $\cF(d',d)$ and $\cF(d'(j), d(j))$ are, respectively, the pullbacks 
\begin{equation*}
 \xymatrixrowsep{1.1pc} \xymatrix{\cF(d',d) \ar[r] \ar[d] & \cD^{d'/}_{/d} \ar[d] & \cF(d'(j),d(j)) \ar[r] \ar[d] & \cD(j)^{d'(j)/}_{/d(j)} \ar[d] \\
 (0_\cM)_{/d} \ar[r] & \cD_{/d} & (0_{\cM(j)})_{/d(j)} \ar[r] & \cD(j)_{/d(j)} }
\end{equation*}
Recall that for any category $\cC$ the forgetful functor $\cC^{c/} \to \cC$ is a left fibration (\cite{HTT} 2.1.2.2), left fibrations are stable under pullbacks, and the total space of left fibration over a space is itself a space (\cite{HTT} 2.1.3.3). Putting this together we see that 
\begin{equation*}
 \cF(d'(j), d(j))^\mathrm{gpd} \simeq \cF(d'(j), d(j)) \quad \cF(d',d)^\mathrm{gpd} \simeq \cF(d',d)
\end{equation*}
and hence the functor $(0_\cM)_{/d} \to \cD_{/d}$ is final. Dualising this argument shows that $0_\cE^{d/} \to \cD^{d/}$ is initial. Therefore Condition \ref{Cond:init} is satisfied.

It remains to show that $\cD$ satisfies Conditions \ref{Cond:bicart}, \ref{Cond:monostab} and \ref{Cond:epistab}. Since an $\infty$-category has limits if and only if it has products and pullbacks, it suffices to consider cases when $\cJ$ is a set indexing a product and when $\cJ$ is the limit diagram for a pullback. In the first case, Conditions \ref{Cond:bicart}, \ref{Cond:monostab} and \ref{Cond:epistab} hold because (co)limit diagrams are computed component-wise in a product. The second case follows from (\cite{HTT} 5.4.5.5). 
\end{proof}

Recall that an $\infty$-category $\cC$ having finite products is {\em cartesian closed} if for each object $c \in \cC$ the functor $c\times -: \cC \to \cC$ has a right adjoint. 

A natural candidate for a right adjoint to $\cA \times -$ for an augmented proto-exact $\infty$-category $\cA$ is the functor $\cFun^\mathrm{ex}(\cA,-)$. However, one must first, for each $\cA,\cB \in \capExact$, endow the $\infty$-category $\cFun^\mathrm{ex}(\cA,\cB)$ with the structure of an augmented proto-exact $\infty$-category. This is done as follows: declare a functor $F: \cA \to \cB$ to be null if $F(a)$ is null for each $a \in A$, and declare a natural transformation $\eta:F \Rightarrow G$ to be an admissible mono- or epi- morphism if all of its components $\eta_a:F(a) \to G(A)$ are such. Since (co)limits are computed object-wise in the target, one can readily verify that this endows $\cFun^\mathrm{ex}(\cA,\cB)$ with an augmented proto-exact $\infty$-structure.

\begin{lem}
 \label{Lem:closed}
 For each augmented proto-exact $\infty$-category $\cA$, the functor
 \begin{equation*}
  \cFun^\mathrm{ex}(\cA,-): \capExact \to \capExact
 \end{equation*}
is right adjoint to $\cA \times -$. Hence, the $\infty$-category $\capExact$ is cartesian closed. 
\end{lem}
\begin{proof}
 Recall that the $\infty$-category $\capExact$ is the simplicial nerve of $\apExact$. It therefore suffices to show that $\cFun^\mathrm{ex}(\cA,-)$ is right adjoint to $\cA\times -$ as functors of simplicially enriched categories (\cite{HTT} 5.2.4.5). That is, we must show that for every triple of augmented proto-exact $\infty$-categories $\cA$, $\cA'$ and $\cA''$ one has a natural isomorphism of simplicial sets
 \begin{equation*}
  \cMap^\mathrm{ex}\left( \cA \times \cA', \cA'' \right) \simeq \cMap^\mathrm{ex} \left( \cA, \cFun^\mathrm{ex}( \cA', \cA'')\right).
 \end{equation*}
 
Forgetting the augmented proto-exact $\infty$-structure one has a natural isomorphism of simplicial sets
 \begin{equation*}
   \xymatrix{ \cFun \left( \cA \times \cA', \cA''\right) \ar[r]^-\Psi_-\simeq & \cFun\left (\cA, \cFun(\cA',\cA'')\right)},
 \end{equation*}
which sends a functor $F:[n]\times\cA\times\cA' \to \cA''$ to the functor 
\begin{equation*}
\Psi F: [n] \times \cA \to \cFun(\cA',\cA'') \quad \Psi F(i,a): a' \mapsto F(i,a,a').
\end{equation*}
If one assumes that $F(i,-,-): \cA \times \cA' \to \cA''$ is exact for each $i \in [n]$ then for each $(i,a) \in [n] \times \cA$, $\Psi F(i,a): \cA' \to \cA''$ is an exact functor and moreover $\Psi F(i,-): \cA \to \cFun^\mathrm{ex}(\cA',\cA'')$ is exact. Therefore the map $\Psi$ descends to a map
\begin{equation*}
 \xymatrix{ \cFun^\mathrm{ex}\left( \cA \times \cA', \cA'' \right) \ar[r]^-\Psi & \cFun^\mathrm{ex} \left( \cA, \cFun^\mathrm{ex}( \cA', \cA'')\right)}
\end{equation*}
which is readily verified to be an isomorphism.
\end{proof}

We can now finally introduce {\em Waldhausen's $S$-construction} \cite{Waldhausen} for augmented proto-exact $\infty$-categories:
\begin{equation}
 \label{Defn:Wald}
 \xymatrixrowsep{.3pc} \xymatrixcolsep{1.8pc} \xymatrix { \capExact \ar[r]^-{S_\bullet -} & \left(\capExact\right)_\Simp \\
 \cA \ar@{|->}[r] & \cFun^\mathrm{ex}(\cArr{\bullet}, \cA) }
 \end{equation}
Said another way, the augmented proto-exact $\infty$-category $S_n \cA$ consists of diagrams
\begin{equation*}
   \xymatrixrowsep{.9pc} \xymatrix{a_{00} \ar[r] & a_{01} \ar[r] \ar[d] & a_{02} \ar[r] \ar[d]& \cdots \ar[r] & a_{0n} \ar[d] \\
     & a_{11} \ar[r] & a_{12} \ar[r]\ar[d] & \cdots \ar[r] & a_{1n} \ar[d] \\
      & & a_{22} \ar[r] & \cdots \ar[r]& a_{2n} \ar[d] \\
      & & &\ddots &\vdots \ar[d] \\
      & & & & a_{nn} }
\end{equation*}
such that the objects $a_{ii}$ are in $\cN$, each horizontal morphism is in $\cM$, each vertical morphism is in $\cE$ and each square is bicartesian. 

\begin{prop}
 \label{Prop:Wald2Seg}
 The $S$-construction $S_\bullet \cA$ of an augmented proto-exact $\infty$-category is a $2$-Segal object in $\capExact$.
\end{prop}
We shall prove this using the {\em path space criterion} for $2$-Segal objects: a simplicial object $\cX$ is $2$-Segal if and only if the {\em initial and final path spaces}
\begin{equation*}
P^{\triangleleft}\cX_\bullet = \cX_{[0]\star[\bullet]} \quad  P^{\triangleright} \cX_\bullet = \cX_{[\bullet]\star[0]},
\end{equation*}
where $\star$ is the join of $\infty$-categories, are Segal objects and for each $n \geq 2$ and $0 \leq i < n$, the image of the square
\begin{equation}
\label{SegUn}
 \xymatrixcolsep{.8pc}\xymatrixrowsep{.8pc}\xymatrix{\Simplex{\{i,i+1\}} \ar[r] \ar[d] & \Simplex n \ar[d] \\
 \Simplex{\{i\}} \ar[r] & \Simplex{n-1}}
\end{equation}
under $\cX$ is a pullback in $\cC$ (\cite{DK12} 6.3.2).

Before giving a proof of Proposition \ref{Prop:Wald2Seg} we must introduce some notation. The functors
\begin{equation*}
 \xymatrixrowsep{.3pc} \xymatrixcolsep{1.5pc} \xymatrix{ [n] \ar[r]^-{H_n} & \cArr {n+1} &  & [n] \ar[r]^-{V_n} & \cArr{n+1} \\
 i \ar@{|->}[r] & [0;i+1] & & i \ar@{|->}[r] & [i;n+1]}
\end{equation*}
define natural transformations of functors $\Simp \to \Catoo$
\begin{equation*}
 H: [\bullet] \Rightarrow \cArr {[0] \star [\bullet]} \quad V: [\bullet] \Rightarrow \cArr{[\bullet] \star [0]}.
\end{equation*}
Pulling back along $H$ and $V$ induce morphisms
\begin{equation}
\label{PathSpCrit}
 \xymatrixcolsep{1.5pc} \xymatrix{ P^\triangleleft S\cA \ar[r]^-{H^*}  & \cFun([\bullet],\cM)  & P^\triangleright S \cA  \ar[r]^-{V^*}  & \cFun([\bullet],\cE)}.
\end{equation}
\begin{proof}[Proof of Prop. \ref{Prop:Wald2Seg}]
Let $S_\bullet := S_\bullet \cA$. The proof is only a slight generalisation of the proof of (\cite{DK12} 7.3.3). 

For each $n$ the morphism $H_n$ can be factored as a sequence of inclusions of full subcategories
\begin{equation*}
 \xymatrixcolsep{1.3pc} \xymatrix{ [n] \ar@{^{(}->}[r]^-{f} & X_n \ar@{^{(}->}[r]^-{g} & Y_n \ar@{^{(}->}[r]^-{h} & \cArr{n+1}}
\end{equation*}
where
\begin{equation*}
 X_n = \{ [0;j]\} \quad Y_n = \{[i;j] \, | \, i=0 \ \mathrm{or} \ i=j \}.
\end{equation*}
By Condition \ref{Cond:init} of Definition \ref{Def:augexact} the essential image of 
\begin{equation*}
\xymatrixcolsep{1.8pc} \xymatrix{\cFun([n], \cM) \ar[r]^-{\iota_\cM\circ -} & \cFun([n], \cA) \ar[r]^-{f_!} & \cFun(X_n, \cA)},
\end{equation*}
where $f_!$ is the left Kan extension along $f$, consists of those functors $F$ with $F(0,0) \in \cN$. Similarly, the essential image of $g_*$, the right Kan extension along $g$, consists of functors satisfying $F(i,i) \in \cN$ and the essential image of $h_!$ functors preserving bicartesian squares of the form Eq.~\ref{APEsquare}. Taken together, this shows that the composite
\begin{equation*}
 h_! g_* f_! (\iota_\cM \circ -): \cFun([n], \cM) \to \cFun^\mathrm{ex}(\cArr{n+1}, \cA)
\end{equation*}
is an equivalence of categories and hence that the first morphism of Eq.~\ref{PathSpCrit} is an equivalence. Dually, one has that the second morphism is an equivalence. Since $\cFun([\bullet],\cM)$ and $\cFun([\bullet],\cE)$ are Segal objects this shows that $P^\triangleleft S$ and $P^\triangleright S$ are Segal objects.

It remains to show that image of the square in Eq.~\ref{SegUn} under $S$,
\begin{equation*}
 \xymatrixcolsep{1.1pc} \xymatrixrowsep{.8pc} \xymatrix{ S_{n-1} \ar[r] \ar[d] & S_n \ar[d] \\
 S_{\{i\}} \ar[r] & S_{\{i,i+1\}} }
\end{equation*}
is a pullback square. Since $\cArr{\{i,i+1\}} \to \cArr{n}$ is a monomorphism of simplicial sets the functor $\cFun(\cArr n , \cA) \to \cFun(\cArr{\{i,i+1\}}, \cA)$ is a categorical fibration (\cite{joyal2008theory} 5.13). Since the conditions defining an exact functor are stable under equivalence this implies that the functor $S_n \to S_{\{i,i+1\}}$ is a categorical fibration. Therefore to show the above square is a pullback it suffices so show the induced map
\begin{equation*}
 S_{n-1} \to S_n \times_{S_{\{i,i+1\}}} S_{\{i\}}
\end{equation*}
is an equivalence, where the target is the {\em ordinary} pullback of simplicial sets. The proof of this is identical to the proof given in (\cite{DK12} 7.3.3). 
\end{proof}

By Lemma \ref{Lem:finlim} the composite functor
\begin{equation*}
 \xymatrixrowsep{.3pc} \xymatrix{\capExact \ar[r]^-{\mathrm{forget}} & \Catoo \ar[r]^-{(-)^\simeq} & \cS}
\end{equation*}
preserves limits. Hence an immediate corollary of Proposition \ref{Prop:Wald2Seg} is the following:
\begin{cor}
 For each $\cA \in \capExact$, the simplicial space $[n] \mapsto \cMap^\mathrm{ex}(\cArr n, \cA)$ is a $2$-Segal space. 
\end{cor}

Lemma \ref{Lem:closed} implies that $S: \capExact \to (\capExact)_\Simp$ preserves limits. In particular, applying $S$ object-wise to a $2$-Segal object in $\capExact$ yields a $2$-Segal object in $(\capExact)_\Simp$, i.e., according to Definition \ref{Defn:double2Seg}, a double $2$-Segal object in $\capExact$. This leads to our two main examples of double $2$-Segal spaces.

The first example is the {\em iterated $S$-construction} of an augmented proto-exact $\infty$-category $\cA$ : The bisimplicial object $S_{\bullet,\bullet} \cA$ in $\capExact$ obtained by applying the $S$-construction object-wise to the $2$-Segal object $S_\bullet \cA$. Explicitly, one defines the iterated $S$-construction to be
\begin{equation}
\label{Defn:IterS}
 S_{\bullet,\bullet}\cA: (\Simpop)^2 \to \capExact, \quad ([n],[k]) \mapsto \cFun^\mathrm{ex}\left(\cArr{n,k}, \cA\right),
\end{equation}
where $\cArr{n,k} = \cArr{[n]\times[k]}$.
\begin{cor}
 \label{Cor:IterS}
 For each $\cA \in \capExact$, the iterated Waldhausen $S$-construction $S_{\bullet,\bullet}\cA$ is a double $2$-Segal object in $\capExact$. Moreover, the bisimplicial space $\cMap^\mathrm{ex}(\cArr{\bullet,bullet},\cA)$ is a double $2$-Segal space.
\end{cor}

The second example is the {\em monoidal $S$-construction} of a monoidal augmented proto-exact $\infty$-category. 
\begin{defn}
\label{Defn:Monexact}
A {\em monoidal augmented proto-exact $\infty$-category} is a simplicial augmented proto-exact $\infty$-category $\cA^\otimes_\bullet \in (\capExact)_\Simp$ satisfying the Segal conditions and having $\cA^\otimes_0 \simeq \ast$.
\end{defn}
\begin{exam}
 Let $A$ be a proto-exact category equipped with a monoidal structure such that the product $\otimes: A \times A \to A$ and unit $\eta: \ast \to A$ are exact. Then the nerve of $A$ is a monoidal augmented proto-exact $\infty$-category. 
\end{exam}
\begin{exam}
 Every exact $\infty$-category equipped with the coproduct monoidal structure is a monoidal augmented proto-exact $\infty$-category.
\end{exam}
The monoidal $S$-construction of $\cA^\otimes_\bullet$, denoted $S^\otimes_{\bullet,\bullet}\cA$, is defined to be the bisimplicial object in $\capExact$ obtained by applying the $S$-construction object-wise to the Segal object $\cA^\otimes_\bullet$. Explicitly,
\begin{equation}
 \label{Defn:MonS}
 S^\otimes_{\bullet,\bullet} \cA: (\Simpop)^2 \to \capExact, \quad ([n],[k]) \mapsto \cFun^\mathrm{ex}\left(\cArr n, \cA^\otimes_k \right).
\end{equation}
\begin{cor}
 \label{Cor:MonS}
 For each monoidal augmented proto-exact $\infty$-category $\cA^\otimes_\bullet$, the monoidal $S$-construction is a double $2$-Segal object in $\capExact$. Moreover, the bisimplicial space $\cMap^\mathrm{ex}(\cArr{\bullet}, \cA^\otimes_\bullet)$ is a double $2$-Segal space.
\end{cor}
\begin{proof}
 Since every Segal object is also a $2$-Segal object (\cite{DK12} 2.5.3, \cite{KockI} 3.5), the monoidal augmented proto-exact $\infty$-category $\cA^\otimes_\bullet$ is a $2$-Segal object in $\capExact$. 
 \end{proof}

\section{Bisimplicial objects define totally lax bialgebras}
\label{Sec:totlax}

As the first step in our construction of the universal Hall bialgebra of a double $2$-Segal space, in this section we shall construct, from a bisimplicial object $\cX \in \cC_{\Simp^2}$ in an $\infty$-category $\cC$ having finite limits, a totally lax bialgebra structure on $\cX_{1,1}$ as an object of $\csmtSpan{\cC}$. That is, we shall construct a symmetric monoidal lax functor $\beta_{\cX}:\csmBialg \rightsquigarrow \csmtSpan{\cC}$ with $\beta_{\cX}(\underline{1}) = \cX_{1,1}$ as per Definition \ref{DefnLaxBialg}. As we show in Section \ref{Sec:SegbiAssoc}, when $\cX$ is a double $2$-Segal space $\beta_\cX$ induces an algebra and a coalgebra structure on $\cX_{1,1}$. The resulting lax bialgebra is the universal Hall bialgebra of $\cX$.

To carry out this construction we exploit the fact that the opposite of the Yoneda embedding,
\begin{equation*}
 \Simplex{\bullet,\bullet}: (\Simpop)^2 \to \ctSetfop,
\end{equation*}
where $\ctSetfop$ is the $\infty$-category of level-wise finite bisimplicial sets, is the {\em initial bisimplicial object} in an $\infty$-category having finite limits: From any bisimplicial object $\cX \in \cC_{\Simp^2}$ in an $\infty$-category having finite limits one has a finite limit preserving functor $\cX:\ctSetfop \to \cC$ given by right Kan extension,
\begin{equation*}
 \xymatrixrowsep{1.1pc} \xymatrix{ \ctSetfop \ar[r]^-\cX & \cC \\
 (\Simpop)^2 \ar@{^{(}->}[u] \ar[ru]_-\cX & }
\end{equation*}
The bisimplicial object $\cX \in \cC_{\Simp^2}$ is then the image of $\Simplex{\bullet,\bullet}$ under the right Kan extension. 

We first, in Section \ref{Sec:bialgcomb}, explicitly construct a symmetric monoidal lax functor
\begin{equation*}
 \xymatrix{\csmBialg \ar@{~>}[r]^-\beta & \csmtSpan{\ctSetfop}}, 
\end{equation*}
endowing the standard $(1,1)$-simplex $\Simplex{1,1}$ with the structure of a totally lax bialgebra. Then, in Section \ref{Sec:intospacesBi}, we show how the object $\cX_{1,1}$ in a bisimplicial object $\cX$ inherits a totally lax bialgebra structure from the universal property of $\Simplex{\bullet,\bullet}$.

In \cite{penney2017simplicial}, we constructed a symmetric monoidal lax functor $\alpha: \csmAlg \rightsquigarrow \csmtSpan{\csSetfop}$ endowing $\Simplex 1$ with the structure of a lax algebra and dually, a symmetric monoidal lax functor $\chi: \csmCoalg \rightsquigarrow \csmtSpan{\csSetfop}$ endowing $\Simplex 1$ with a lax coalgebra structure. By right Kan extension along the functors 
\begin{equation*}
\Simplex{\bullet,1}, \ \Simplex{1,\bullet}: \Simpop \to \ctSetfop
\end{equation*}
one obtains, respectively, finite limit preserving functors
 \begin{equation*}
  (-)^h, (-)^v: \csSetfop \to \ctSetfop, \quad X^h:= X \boxtimes \Simplex 1, X^v  = \Simplex 1 \boxtimes X.
 \end{equation*}
 The symmetric monoidal lax functor $\beta$ will be an extension of both $\alpha$ and $\chi$ in the sense that one has a commutative diagram
 \begin{equation*}
  \xymatrixrowsep{.9pc} \xymatrix{\csmAlg\ar@{~>}[r]^-{\alpha} \ar[d]_-{\iota_a} & \csmtSpan{\csSetfop} \ar[d]^-{(-)^h} \\
  \csmBialg \ar@{~>}[r]_-{\beta} & \csmtSpan{\ctSetfop} \\
  \csmCoalg \ar[u]^-{\iota_c} \ar@{~>}[r]_-\chi  & \csmtSpan{\csSetfop} \ar[u]_-{(-)^v} }
 \end{equation*}

 \paragraph{Informal description of the symmetric monoidal lax functor $\beta$.} On objects, $\beta$ is given by
 \begin{equation*}
  \beta: X \mapsto X\cdot \Simplex{1,1} = \coprod_{x \in X} \Simplex{1,1}.
 \end{equation*}
The image under $\beta$ of a span $f:\xymatrixcolsep{1.2pc} \xymatrix{ X_0 & \ar[l]_-p Y \ar[r]^-q & X_1}$ is the span in $\ctSetfop$ given by the outer edges of the diagram
\begin{equation}
\label{betainfone}
 \xymatrixrowsep{.7pc} \xymatrixcolsep{1.2pc} \xymatrix{ & & \beta(f) & & \\
 & \alpha(p)^v \ar[ru] & & \alpha(q)^h \ar[lu] & \\
 \beta(X_0)\ar[ru] & &\beta(Y)\ar[ru] \ar[lu] & & \beta(X_1) \, \ar[lu] }
\end{equation}
where the middle square is a pushout.

The lax structure for $\beta$ is considerably more difficult to describe. Let us look at some special cases. First, consider a pair of composable morphisms in the image of $\csmAlg$, that is, a composite of the form
\begin{equation*}
 \xymatrixrowsep{.7pc} \xymatrixcolsep{1.2pc} \xymatrix { & & X_0  \ar@{=}[ld] \ar[rd]^-{p_1} & & \\
 & X_0 \ar@{=}[ld] \ar[rd]^-{p_1} & & X_1 \ar@{=}[ld] \ar[rd]_-{p_2} & \\
 X_0 & & X_1 & & X_2}
\end{equation*}
As we wish for $\beta$ to extend $\alpha$, the component of $\beta$'s lax structure for this composite will be the diagram
\begin{equation*}
 \xymatrixrowsep{.7pc}\xymatrixcolsep{1.2pc}\xymatrix{ & \alpha(p_2p_1)^h \ar@{=}[d] & \\
 \beta(X_0) \ar[ru] \ar[rd] \ar[r] & \alpha(p_2p_1)^h & \beta(X_2) \, . \ar[lu] \ar[ld] \ar[l] & \\
  & \alpha(p_2)^h \coprod_{\beta(X_1)} \alpha(p_1)^h \ar[u] & }
\end{equation*}
For a composite in the image of $\csmCoalg$ one does the same, replacing $(-)^h$ for $(-)^v$.

Next, consider a pair of composable morphisms $p$ and $q$ given by the composite
\begin{equation*}
 \xymatrixrowsep{.7pc} \xymatrixcolsep{1.2pc} \xymatrix { & & X_0 \times_{X_1} X_2 \ar[ld]_-{r_0} \ar[rd]^-{r_1} & & \\
 & X_0 \ar@{=}[ld] \ar[rd]^-{p} & & X_2 \ar@{=}[rd] \ar[ld]_-{q} & \\
 X_0 & & X_1 & & X_2}
\end{equation*}
Then the component of $\beta$'s lax structure for this composite will be a diagram
\begin{equation*}
 \xymatrixrowsep{.7pc} \xymatrixcolsep{1.2pc} \xymatrix{ & \beta(q \circ p) \ar[d] & \\
 \beta(X_0) \ar[ru] \ar[r] \ar[rd] & \fB_{q,p} & \beta(X_2)\, , \ar[lu] \ar[l] \ar[ld] \\
 & \beta(q) \coprod_{\beta(X_1)} \beta(p) \ar[u] & }
\end{equation*}
where $\fB_{q,p}$ is the bisimplicial set 
\begin{equation*}
 \fB_{q,p} = \coprod_{x \in X_1} \Simplex{|p^{-1}(x)|, |q^{-1}(x)|}.
\end{equation*}

\begin{exam}
 For the composite
 \begin{equation*}
 \xymatrixrowsep{.7pc} \xymatrixcolsep{1.2pc} \xymatrix { & & \underline{4} \ar[ld]_-{m^2} \ar[rd]^-{\overline{m}^2} & & \\
 & \underline{2} \ar@{=}[ld] \ar[rd]^-{m} & & \underline{2} \ar@{=}[rd] \ar[ld]_-{m} & \\
 \underline{2} & & \underline{1} & & \underline{2}}
\end{equation*}
the lax structure of $\beta$ is the bisimplicial set $\Simplex{2,2}$, which is exactly the witness for the lax compatibility of the product and coproduct in the universal Hall algebra of a proto-exact category that we described in Eq.~\ref{laxbialginf}.
\end{exam}

For a general composition, the component of the lax structure will be given as a certain colimit of the special cases considered above. Most of the difficulty in giving the rigorous construction of $\beta$ comes from the need to give a coherent parametrisation of these colimit diagrams for all strings of composable morphisms in $\csmBialg$.

\paragraph{Outline of the construction.} According to Definitions \ref{DefnsmLax} and \ref{DefnLaxBialg}, a totally lax bialgebra structure on $\Simplex{1,1}$ as an object of $\csmtSpan{\ctSetfop}$ is a functor
 \begin{equation*}
  \xymatrixrowsep{.8pc}\xymatrixcolsep{.9pc}\xymatrix{\Un(\csmBialg)\ar[rr]^-{\beta} \ar[rd] & & \csmtSpan{\ctSetfop} \ar[ld] \\
  & \cFinDel & }
 \end{equation*}
which preserves cocartesian lifts of morphisms in $\cFin\times\Simpin^{\rm op}$. Furthermore, the image of the object $\underline{1}$ in the fibre over $(\underline{1}_\ast,[0])$ must be $\Simplex {1,1}$.

One can readily see that $\Un(\csmBialg)$ is the nerve of the ordinary Grothendieck construction of the functor
 \begin{equation*}
  \FinDel \to \Cat, \quad (S_*,[n]) \mapsto \Bialg_{S,n}.
 \end{equation*}
That is, an element of $\Un(\csmBialg)_k$ is a triple $((f,\phi), \theta, \gamma)$, where:
\begin{itemize}
 \item  $(f,\phi)$ is an element of  $N_k(\FinDel)$;
 \item $\theta$ is a family of functors
 \begin{equation*}
 \theta_i:\Cart{f(i)}\times\Pyrnc{\phi(i)} \to \Alg \in \Bialg_{f(i),\phi(i)}, \quad i \in [k];
 \end{equation*}
 
 \item $\gamma$ is a family of natural isomorphisms 
 \begin{equation*}
 \gamma_i:\xymatrixcolsep{1.1pc} \xymatrix{(f_i,\phi_i)_*\theta_i\ar@{=>}[r]^-\sim &  \theta_{i+1}}, \quad i = 0, \ldots, k-1.
 \end{equation*}
\end{itemize}

From such data we will build a cartesian, vertically constant functor
\begin{equation*}
 \beta\left( \theta, \gamma\right): \Cartop{f(0)} \times \Pyr{M_\phi} \to \tSetfop.
\end{equation*}
By Proposition \ref{UnstrSimps} this defines a $k$-simplex in the unstraightening of $\csmtSpan{\ctSetfop}$. As the general construction of $\beta(\theta,\gamma)$ will be somewhat involved, let us first consider a very special case.

Consider $\left((f,\phi),\theta,\gamma\right) \in \Un(\csmBialg)_1$, where $f$ is constant on $\underline{1}_\ast$ , $\phi:[2] \actmorL [1]$ is the unique active morphism and the natural isomorphism $\gamma_0$ is the identity. Then the family $\theta$ is determined by a pair of composable morphisms $p$ and $q$ in $\csmBialg$, that is, a diagram in $\Alg$
\begin{equation*}
 \xymatrixrowsep{.7pc} \xymatrixcolsep{1.2pc} \xymatrix{ & & Z \ar[ld]_-{r_0} \ar[rd]^-{r_1} & & \\
 & Y_0 \ar[ld]_-{p_0} \ar[rd]^-{p_1} & & Y_1 \ar[ld]_-{q_0} \ar[rd]^-{q_1} & \\
 X_0 & & X_1 & & X_2 }
\end{equation*}
where the middle square is a pseudo-pullback. Since $\beta(\theta,\gamma)$ is cartesian it is determined by its restriction to $\Wedge{M_\phi}$, the diagram described in Example \ref{Vphiact}. The restriction of $\beta(\theta,\gamma)$ to $\Wedge{M_\phi}$ is
\begin{equation*}
 \xymatrixrowsep{.8pc} \xymatrixcolsep{1.1pc} \xymatrix{\beta(X_0) \ar[rr] \ar@{=}[d]  & & \beta(qp) \ar[d] & & \beta(X_2) \ar[ll] \ar@{=}[d] \\
  \beta(X_0) \ar[rr] & & \fB_{q,p} & & \beta(X_2)\ar[ll] \\
  \beta(X_0) \ar@{=}[u] \ar[r] & \beta(p) \ar[ru] & \beta(X_1) \ar[l] \ar[r] & \beta(q) \ar[lu] & \beta(X_2) \ar[l] \ar@{=}[u] }
\end{equation*}
The bisimplicial set $\fB_{q,p}$ is the colimit of the diagram
\begin{equation*}
 \xymatrixrowsep{.4pc}\xymatrixcolsep{.6pc} \xymatrix{ & & & & \beta(Z) \ar[rd] \ar[ld] & & & & \\
 & \alpha(p_0r_0)^v & & \ar[ll] \alpha(r_0)^v \ar[rd] & &  \ar[ld]\alpha(r_1)^h\ar[rr] & & \alpha(q_1r_1)^h & \\
 & & \beta(Y_0) \ar[ru] \ar[rd] \ar[ld] & & \fB_{q_1,p_1} & & \beta(Y_1) \ar[lu] \ar[ld] \ar[rd] & & \\
 & \ar[uu]\alpha(p_0)^v & & \alpha(p_1)^h \ar[ru] & & \alpha(q_0)^v \ar[lu] & & \alpha(q_1)^h \ar[uu] & \\
 \beta(X_0) \ar[ru] & & & & \beta(X_1) \ar[ru] \ar[lu] & & & & \beta(X_2) \ar[lu]}
\end{equation*}
This diagram is a functor $\overline{\beta}(\theta,\gamma)$ from the twisted arrow category of $\Pyr{2}$ to $\tSetfop$. 

Looking back at the description of $\beta$ on $1$-morphisms in Eq.~\ref{betainfone}, one observes that {\em every} object in the restriction of $\beta(\theta,\gamma)$ to $\Wedge{M_\phi}$ is a colimit over some subdiagram of $\overline{\beta}(\theta,\gamma)$. We therefore refer to $\overline{\beta}(\theta,\gamma)$ as the {\em master diagram}, as the functor $\beta(\theta,\gamma)$ is obtained from it by systematically extracting colimits over subdiagrams.

For a general $((f,\phi),\theta,\gamma)\in \Un(\csmBialg)_k$, the master diagram is a functor
\begin{equation}
\label{Masterfirst}
 \overline{\beta}(\theta,\gamma): \Singop{f(0)} \times \Pyr{\Omega_\phi} \to \tSetfop,
\end{equation}
where $\Omega_\phi$ is a suitable generalisation of $\Pyr{2}$ that we shall define in Section \ref{Sec:bialgcomb}.

The process of systematically extracting colimits over subdiagrams of the master diagram $\overline{\beta}(\theta,\gamma)$ is formalised in the construction of the {\em cone functor}
\begin{equation}
\label{Conefirst}
 \kappa_\phi: \Pyr{M_{\phi}} \to \widehat{\Pyr{\Omega_\phi}},
\end{equation}
where $\widehat{\Pyr{\Omega_\phi}}$ is the free completion of $\Pyr{\Omega_\phi}$. 

The functor $\beta(\theta,\gamma)$ is defined by a sequence of right Kan extensions and restrictions as described by the following diagram
\begin{equation}
\label{ResRKE}
 \xymatrixrowsep{.9pc} \xymatrix{ \Cartop{f(0)}\times\Pyr{M_\phi} \ar[rrr]^-{\beta(\theta,\gamma)} & & & \tSetfop \\
 \Singop{f(0)}\times \Wedge{M_\phi} \ar@{^{(}->}[u] \ar@{^{(}->}[r] & \Singop{f(0)} \times\Pyr{M_\phi} \ar[r]_{\id \times \kappa_\phi} & \Singop{f(0)} \times \widehat{\Pyr{\Omega_\phi}} \ar[r] & \tSetfop \ar@{=}[u] \\
 & & \Singop{f(0)}\times \Pyr{\Omega_\phi} \ar@{^{(}->}[u] \ar[r]_-{\overline{\beta}(\theta,\gamma)} & \tSetfop \ar@{=}[u] }
\end{equation}
where each square is a right Kan extension.

\subsection{The totally lax bialgebra structure on \texorpdfstring{$\Simplex{1,1}$}{the (1,1)-simplex}}
\label{Sec:bialgcomb}
As outlined above, the construction of $\beta$ proceeds by first defining, for a given $((f,\phi),\theta,\gamma) \in \Un(\csmBialg)_k$, the master diagram of Eq.~\ref{Masterfirst} and cone functor of Eq.~\ref{Conefirst}. The image of $((f,\phi),\theta,\gamma)$ under $\beta$ is then defined by a sequence of right Kan extensions and restrictions according to the diagram in Eq.~\ref{ResRKE}. 

Before going further we must define the categories $\Omega_\phi$ which appear crucially at all stages of the construction. For a given $\phi \in N_k(\Delta^{\rm op})$, define $\Omega_\phi$ to be the poset having object set
\begin{equation*}
 \left\{ ([i;j],b) \ | \ b \in [k]^{\rm op}, \ [i;j] \in \Pyr{\phi(b)} \right\},
\end{equation*}
and ordering defined by declaring $([i;j],b) \leq ([i';j'],b')$ if and only if $b \leq b' \in [k]^{\rm op}$ and $\phi_{b,b'}([i;j]) \leq [i';j'] \in \Pyr{\phi(b')}$. In fact, $\Omega_\phi$ is the ordinary Grothendieck construction of the functor
\begin{equation*}
 \xymatrix{[k]^{\rm op} \ar[r]^-\phi & \Delta \ar[r]^-{\Pyr{\bullet}} & \Cat}.
\end{equation*}
 Hence, for a natural transformation $\eta: \phi' \Rightarrow \phi$ one has a functor 
\begin{equation*}
\Omega(\eta): \Omega_{\phi'} \to \Omega_{\phi}, \ ([i;j],b) \mapsto ([\eta_b(i); \eta_b(j)], b)
\end{equation*}
and for a morphism $\gamma: [n] \to [k]$ there is a functor 
\begin{equation*}
\Omega(\gamma): \Omega_{\phi \gamma} \to \Omega_{\phi}, \ ([i;j],b) \mapsto ([i;j],\gamma(b)).
\end{equation*}
Finally, observe that $V_\phi$ is a subposet of $\Omega_\phi$ according to the functor
\begin{equation*}
 \nu_\phi: (a,b) \mapsto ([a;a],b).
\end{equation*}

\begin{exam}
For $\phi \in N_k(\Delta^{\rm op})$ the constant map on $[n]$, the poset $\Omega_\phi$ is isomorphic to $\Pyr{n}\times[k]$.
\end{exam}
\begin{exam}
  For the unique active morphism $\phi = ([2] \actmorL [1]) \in N_1(\Delta^{\rm op})$, the poset $\Omega_\phi$ is
  \begin{equation*}
   \xymatrixrowsep{.4pc} \xymatrixcolsep{.9pc} \xymatrix{([0;0],1) \ar[dd] & & ([0;1],1) \ar[rr] \ar[ll] \ar[dd] & & ([1;1],1) \ar[dd] \\
    & & & & \\
    ([0;0],0) & & ([0;2],0) \ar[rr] \ar[ll] \ar[rd] \ar[ld] & & ([2;2],0) \\
    & ([0;1],0) \ar[lu] \ar[rd] & & ([1;2],0) \ar[ru] \ar[ld] & \\
    & & ([1;1],0) & & }
  \end{equation*}
  \end{exam}
  \begin{exam}
  For $\phi \in N_1(\Delta^{\rm op})$ the constant map on $[1]$, the poset $\Pyr{\Omega_\phi}$ is
  \begin{equation*}
   \xymatrixrowsep{.8pc} \xymatrixcolsep{1.4pc} \xymatrix{ \bullet & \bullet \ar[l] \ar[r] & \bullet & \bullet \ar[l] \ar[r] & \bullet \\
   \bullet \ar[u] \ar[d] & \bullet \ar[l] \ar[r] \ar[u] \ar[d] & \bullet \ar[u] \ar[d] & \bullet \ar[l] \ar[r] \ar[u] \ar[d] & \bullet \ar[u] \ar[d] \\
   \bullet & \bullet \ar[l] \ar[r] & \bullet & \bullet \ar[l] \ar[r] & \bullet} 
  \end{equation*}
 \end{exam}
 
 We are now ready to proceed with the step-by-step construction of the symmetric monoidal lax functor $\beta: \csmBialg \rightsquigarrow \csmtSpan\ctSetfop$. 
 
 \paragraph{Construction of the master diagram.} Fix $((f,\phi), \theta,\gamma) \in \Un(\csmBialg)_k$. We shall construct, according to Remark \ref{Rem:oplax}, the master diagram $\overline{\beta}:=\overline{\beta}(\theta,\gamma)$ of Eq.~\ref{Masterfirst} as a normal oplax functor
\begin{equation*}
 \overline{\beta}:\Singop{f(0)} \times \Omega_\phi \nrightarrow \sp{\tSetfop}.
\end{equation*}

Recall from Eq.~\ref{nabladefn} that $\nabla$ is the category of spans of the form $\xymatrixcolsep{.7pc} \xymatrix{ \langle n \rangle & \ar[l] \langle k \rangle \ar@{>->}[r] & \langle m \rangle}$ in $\Delta_+$. The category of {\em levelwise finite bi-$\nabla$ sets}, denoted $\Fin_{\nabla^2}$ is the category of functors $(\nabla^\op)^{\times 2} \to \Fin$. From a pseudo-pullback square in $\Alg$
\begin{equation*}
 \xymatrixrowsep{.9pc}\xymatrix{X' \ar[r]^-{q_2} \ar[d]_-{q_1} & Y' \ar[d]^-{p_2} \\
 X \ar[r]_-{p_1} & Y}
\end{equation*}
one can define the following diagram of levelwise finite bi-$\nabla$ sets:
\begin{equation*}
 \xymatrixrowsep{.9pc} \xymatrix{ & \displaystyle \coprod_{y \in Y} \Simplexn{ p_1^{-1}(y), p_2^{-1}(y)} & \\
 X' \ar[ru] & & Y \ar[lu] }
\end{equation*}
where $X'$ and $Y$ are constant bi-$\nabla$ sets and $\Simplexn{n,m}$ is the functor represented by $(\langle n \rangle, \langle m\rangle)$. The first map arises from the following morphism in $\nabla^{\times 2}$
\begin{equation*}
 \xymatrixrowsep{.7pc} \xymatrix{ & ( \{x'\},\{x'\}) \ar@{=}[ld]\ar@{>->}[rd] & \\
 ( \{x'\},\{x'\}) & & (p_1^{-1}(p_1q_1(x')), p_2^{-1}(p_2q_2(x'))) \, ,}
\end{equation*}
and the second morphism from
\begin{equation*}
 \xymatrixrowsep{.7pc} \xymatrix{  & (p_1^{-1}(y), p_2^{-1}(y) ) \ar[ld]\ar@{=}[rd] & \\
 (\{y\}, \{y\}) & & (p_1^{-1}(y), p_2^{-1}(y) ) \, .}
\end{equation*}
 We can then apply the functor $\fG^*$ of Eq.~\ref{Augdef} to obtain a span in $\tSetfop$.

Now let $f:(s,[i;j],b) \to (s,[i';j'],b')$ be a morphism in $\Singop{f(0)}\times\Omega_\phi$. Then one has the following diagram in $\Pyrnc{\phi(b')}$
\begin{equation}
\label{morphm1}
 \xymatrixrowsep{.9pc} \xymatrix{ [\phi_{b,b'}(i);\phi_{b,b'}(j)] \ar[r] \ar[d] & [i';\phi_{b,b'}(j)] \ar[d] \\
 [\phi_{b,b'}(i);j'] \ar[r] & [i';j'] }
\end{equation}
The image of this square under $\theta_{b'}(f_{b',0}^{-1}(s),-)$ is, by assumption, a pseudo-pullback square in $\Alg$. Define $\overline{\beta}$ on objects to be
\begin{equation*}
 \overline{\beta}_b(s,[i;j]) = \fG^* \left(\theta_b (f^{-1}_{b,0}(s),[i;j])\right) = \theta_b (f^{-1}_{b,0}(s),[i;j]) \cdot \Simplex{1,1}.
\end{equation*}
We can apply the preceding construction to produce a span in $\tSetfop$
\begin{equation*}
 \xymatrixrowsep{.9pc} \xymatrix{ & \overline{\beta}(f) & \\
 \overline{\beta}_{b'}(s,[\phi_{b,b'}(i);\phi_{b,b'}(j)]) \ar[ru] & & \ar[lu] \overline{\beta}_{b'}(s, [i';j']) \ . }
\end{equation*}
The image of $f$ under $\overline{\beta}$ is the span obtained by precomposing the left leg with the natural isomorphism $\gamma_{b,b'}$
\begin{equation*}
 \xymatrixrowsep{.9pc} \xymatrix{ & \overline{\beta}(f) & \\
 \overline{\beta}_b(s,[i;j]) \ar[ru] & & \ar[lu] \overline{\beta}_{b'}(s, [i';j']) }
\end{equation*}

Next, consider a diagram in $\Alg$ of the form
\begin{equation*}
 \xymatrixrowsep{.9pc} \xymatrix{ X'' \ar[r] \ar[d] & Y'' \ar[d]^-{r_2} \ar[r]^-{s_2} & Z'' \ar[d]^-{q_2} \\
 X' \ar[r]_-{r_1} \ar[d]_-{s_1} & Y' \ar[d]^-{t_1} \ar[r]^-{t_2} & Z' \ar[d]^-{p_2} \\
 X \ar[r]_-{q_1} & Y \ar[r]_-{p_1} & Z}
\end{equation*}
where each square is a pseudo-pullback. From this one defines the following commutative diagram bi-$\nabla$ sets
\begin{equation*}
 \xymatrixrowsep{.9pc} \xymatrixcolsep{.5pc}\xymatrix{ &  \displaystyle \coprod_{z \in Z} \Simplexn{ (p_1q_1)^{-1}(z), (p_2q_2)^{-1}(z)} & & & \\
 X'' \ar[r] \ar[ru] & \displaystyle \coprod_{y \in Y'} \Simplexn{ r_1^{-1}(y),  r_2^{-1}(y)}  \ar[u]_-a & \ar[l]  Y'\ar[r]  & \displaystyle \coprod_{z \in Z} \Simplexn{p_1^{-1}(z), p_2^{-1}(z)} \ar[llu]^-b & \ar[l] \ar@/_1.7pc/[lllu] Z \, , }
\end{equation*}
where the morphism $a$ is induced by $s_1$ and $s_2$ and the morphism $b$ is induced by $q_1$ and $q_2$. Applying the functor $\fG^*$ of Eq.~\ref{Augdef} yields a diagram in $\tSetf$.

We will now define the oplax structure on $\overline{\beta}$. Let $\xymatrixcolsep{1pc} \xymatrix{ (s,[i;j],b) \ar[r]^-{f} & (s,[i';j'],b') \ar[r]^-g & (s,[i'';j''],b'')}$ be a pair of composable morphisms in $\Singop{f(0)}\times \Omega_\phi$. Then one has the following diagram in $\Pyrnc{\phi(b'')}$
\begin{equation}
\label{morphm2}
 \xymatrixrowsep{.9pc} \xymatrix{[\phi_{b,b''}(i);\phi_{b,b''}(j)] \ar[r] \ar[d] & [\phi_{b',b''}(i');\phi_{b,b''}(j)] \ar[r] \ar[d] & [i'';\phi_{b,b''}(j)] \ar[d] \\
 [\phi_{b,b''}(i);\phi_{b,b''}(j')] \ar[r] \ar[d] & [\phi_{b',b''}(i'); \phi_{b',b''}(j')] \ar[d]\ar[r] & [i'';\phi_{b',b''}(j')] \ar[d]\\
 [\phi_{b,b''}(i);j''] \ar[r] & [\phi_{b',b''}(i');j''] \ar[r] & [i'';j'']}
\end{equation}
Using the preceding discussion and the natural isomorphisms $\gamma_{b,b''}$ and $\gamma_{b',b''}$ we obtain a commutative diagram
\begin{equation}
\label{BetaOplaxPush}
 \xymatrixrowsep{.9pc} \xymatrix{ \overline{\beta}_{b'}(s',[i';j']) \ar[r] \ar[d] & \overline{\beta}(g) \ar[d] \\
 \overline{\beta}(f) \ar[r] & \overline{\beta}(g \circ f) \ . }
\end{equation}
We define the corresponding component of the oplax structure,
\begin{equation*}
\overline{\fB}_{g,f}:\overline{\beta}(g) \coprod_{\overline{\beta}_{b'}(s',[i';j'])} \overline{\beta}(f) \to \overline{\beta}(g \circ f),
\end{equation*}
to be the universal morphism induced by Eq.~\ref{BetaOplaxPush}.

\begin{lem}
 For each $\left((f,\phi), \theta,\gamma\right) \in \Un(\csmBialg)_k$, the data given above defines a normal oplax functor and hence, by Remark \ref{Rem:oplax}, a functor
 \begin{equation*}
  \overline{\beta}(\theta,\gamma): \Singop{f(0)}\times \Pyr{\Omega_\phi} \to \tSetfop.
 \end{equation*}
\end{lem}
\begin{proof}
 To see that $\overline{\fB}_{\id, g} = \id_{\overline{\beta}(g)}=\overline{\fB}_{g, \id}$ one need only note that the $2\times 2$ grids of pseudo-pullbacks in $\Alg$ defining these components take, respectively, the forms
 \begin{equation*}
  \xymatrixrowsep{1.1pc} \xymatrixcolsep{1.1pc} \xymatrix{\bullet \ar[r] \ar[d] & \bullet \ar[d] \ar@{=}[r] & \bullet \ar[d] &  & \bullet \ar@{=}[r] \ar@{=}[d]& \bullet \ar[r]\ar@{=}[d] & \bullet \ar@{=}[d] \\
  \bullet \ar[r] \ar@{=}[d] & \bullet \ar@{=}[d] \ar@{=}[r] & \bullet \ar@{=}[d] & {\rm and} & \bullet \ar[d] \ar@{=}[r] & \bullet \ar[r] \ar[d] & \bullet \ar[d] \\
  \bullet \ar[r] & \bullet \ar@{=}[r] & \bullet & & \bullet \ar@{=}[r] & \bullet \ar[r] & \bullet }
 \end{equation*}
Given a triple $\xymatrixcolsep{1.1pc} \xymatrix{w \ar[r]^-g & x \ar[r]^-h & y \ar[r]^-i & z}$ of composable morphisms in $\Singop{f(0)}\times\Omega_\phi$ one has, in the same manner as above, a $3 \times 3$ grid of pseudo-pullbacks in $\Alg$. From this one obtains a diagram $\tSetf$
\begin{equation*}
 \xymatrixrowsep{.7pc} \xymatrix{ & & & \overline{\beta}(ihg) & & & \\
 & & \overline{\beta}(gh) \ar[ru] & & \overline{\beta}(ih) \ar[lu] & & \\
 & \overline{\beta}(g) \ar[ru] & & \overline{\beta}(h) \ar[ru] \ar[lu] & & \overline{\beta}(i) \ar[lu] & \\
 \overline{\beta}(w) \ar[ru] & & \overline{\beta}(x) \ar[ru] \ar[lu] & & \overline{\beta}(y) \ar[ru] \ar[lu] & & \overline{\beta}(z) \ar[lu] }
\end{equation*}
Then one has that
\begin{equation*}
 \overline{\fB}_{i,hg} \circ \left(\id_{\overline{\beta}(i)} \coprod_{\overline{\beta}(y)} \overline{\fB}_{h,g}\right)  = \overline{\fB}_{ih,g} \circ  \left( \overline{\fB}_{i,h} \coprod_{\overline{\beta}(x)} \id_{\overline{\beta}(g)}\right)
\end{equation*}
since both the left and right hand side are universal morphisms for the same colimit. Namely, the colimit over the bottom two rows of the preceding diagram. 
\end{proof}

 The functor $\overline{\beta}$ is vertically constant in the following sense.
\begin{lem}
 \label{MasterVert}
 For each $\left((f,\phi), \theta,\gamma\right) \in \Un(\csmBialg)_k$, the image of every morphism in the composite functor
 \begin{equation*}
  \xymatrix{ \Singop{f(0)}\times \Pyr{V_\phi} \ar[r]^-{\Pyr{\nu_\phi}} & \Singop{f(0)}\times \Pyr{\Omega_\phi}\ar[r]^-{\overline{\beta}(\theta,\gamma)} & \tSetfop}
 \end{equation*}
is an isomorphism. 
\end{lem}
\begin{proof}
 By Remark \ref{Rem:oplax} it suffices to show that for the normal oplax functor
 \begin{equation*}
  \xymatrix{ \Singop{f(0)}\times V_\phi \ar[r]^-{\nu_\phi} & \Singop{f(0)}\times\Omega_\phi \ar[r]^-{\overline{\beta}(\theta,\gamma)} & \sp{\tSetfop}}
 \end{equation*}
 both the images of morphisms and the components of the oplax structure are isomorphisms. This follows from observing that when restricted to $V_\phi$ the diagrams \ref{morphm1} and \ref{morphm2} defining, respectively, the images of morphisms and components of the oplax structure, are constant. 
\end{proof}

Before moving on to the construction of the cone functor, let us consider a simple example of the master diagram that will reappear later in this section.

\begin{exam}
\label{MasterOne}
 Consider the case of $(f,\phi) \in N_1(\FinDel)$ with $f$ the constant map on $\underline{1}_\ast$ and $\phi$ the constant map on $[1]$. Then an element $((f,\phi),\theta,\gamma)\in \Un(\csmBialg)_1$ is the data of a diagram 
 \begin{equation*}
  \xymatrixrowsep{1.1pc} \xymatrix{ X_0 \ar[d]^-{\wr}_-{\gamma_0} & \ar[l]_-{p_0} Y \ar[d]_-{\wr}^-{\gamma} \ar[r]^-{p_1} & X_1 \ar[d]_-{\wr}^-{\gamma_1} \\
  X_0' & \ar[l]_-{p_0'} Y' \ar[r]^-{p_1'} & X_1'\, ,}
 \end{equation*}
 in $\Alg$ and $\overline{\beta}$ is
 \begin{equation*}
  \xymatrixrowsep{1.1pc} \xymatrixcolsep{1.1pc} \xymatrix{\overline{\beta}(X_0) \ar[d]^-\wr \ar[r] & \overline{\beta}(p_0) \ar[d]^-\wr & \overline{\beta}(Y) \ar[l] \ar[d]^-\wr \ar[r] & \overline{\beta}(p_1) \ar[d]^-\wr & \overline{\beta}(X_1) \ar[l] \ar[d]^-\wr \\
  \overline{\beta}(\gamma_0) \ar[r] & \overline{\beta}(\gamma_0p_0) & \overline{\beta}(\gamma) \ar[l] \ar[r] & \overline{\beta}(\gamma_1p_1)  & \overline{\beta}(\gamma_1) \ar[l]  \\
  \overline{\beta}(X_0') \ar[u]_-\wr \ar[r] & \overline{\beta}(p_0') \ar[u]_-\wr & \overline{\beta}(Y') \ar[l] \ar[u]_-\wr \ar[r] & \overline{\beta}(p_1') \ar[u]_-\wr & \overline{\beta}(X_1') \ar[l] \ar[u]_-\wr} 
 \end{equation*}
\end{exam}

\paragraph{Construction of the cone functor.} Fix $\phi \in N_k(\Delta^\op)$. We shall construct, as per Remark \ref{Rem:oplax}, the cone functor of Eq.~\ref{Conefirst} as a normal oplax functor
\begin{equation*}
 \kappa_\phi: M_\phi \nrightarrow \sp{\widehat{\Pyr{\Omega_\phi}}},
\end{equation*}
where $\widehat{\Pyr{\Omega_\phi}}$ is the free completion of $\Pyr{\Omega_\phi}$, which we now define.

The {\em free completion} of a poset $Q$, denoted $\widehat{Q}$, is the opposite of its poset of upwards-closed subsets, that is, 
\begin{equation*}
 \widehat{P} = \Fun(P, [1])^{\rm op}.
\end{equation*}
The free completion $\widehat{Q}$ satisfies the following universal property: any functor $F: Q \to D$, for $D$ complete, uniquely factors as
\begin{equation*}
 \xymatrixrowsep{.9pc} \xymatrix{ \hat{Q} \ar[r]^-{\hat{F}} & D\, , \\
 Q \ar@{^{(}->}[u]^-\uparrow \ar[ru]_-F & }
\end{equation*}
where 
\begin{equation*}
 \uparrow\, :\xymatrixcolsep{1.5pc}\xymatrix{Q \ar@{^{(}->}[r]& \widehat{Q}}, \quad q \mapsto \{ q' \in Q \ | \ q'\geq q\}.
\end{equation*}

The definition of the cone functor $\kappa_\phi$ is based on the following observation. To each interval $[(a,b);(a',b')]$ in $M_\phi$ one can associate the subposet of $\Omega_\phi$,
\begin{equation*}
 \overline{\kappa}_\phi [(a,b);(a',b')] = \left\{ ([i;j],c) \ | \ b \leq c \leq b', \, \phi_{b,c}(a)\leq i, \, \phi_{c,b'}(j) \leq a' \right\}.
\end{equation*}
These subposets are such that for $[x;x'] \subset [y;y']$ one has that $\overline{\kappa}_\phi[x;x'] \subset \overline{\kappa}_\phi[y;y']$. 
\begin{exam}
 Let $\phi:[2] \actmorL [1]$ be the unique active morphism and consider the following two intervals, coloured in red and blue, in $M_\phi$
 \begin{equation*}
 \xymatrixrowsep{.8pc}\xymatrixcolsep{1.1pc} \xymatrix{{\color{red} (0,1)} \ar[d] {\color{red} \ar@[red][rr]} & & {\color{red} (1,1)} \ar[d] \\
 {\color{blue} (0,0)} {\color{blue} \ar@[blue][r]} & {\color{blue} (1,0)} {\color{blue} \ar@[blue][r]} & {\color{blue} (2,0)} }
\end{equation*}
The associated subposets of $\Omega_\phi$ defined by $\overline{\kappa}_\phi$ are the following
\begin{equation*}
   \xymatrixrowsep{.4pc} \xymatrixcolsep{.9pc} \xymatrix{{\color{red} ([0;0],1) \ar[dd]} & & {\color{red} ([0;1],1) \ar@[red][rr]} \ar@[red][ll] \ar[dd] & & {\color{red} ([1;1],1)} \ar[dd] \\
    & & & & \\
    {\color{blue} ([0;0],0)} & & {\color{blue} ([0;2],0) \ar@[blue][rr] \ar@[blue][ll] \ar@[blue][rd] \ar@[blue][ld]} & & {\color{blue} ([2;2],0)}  \\
    & {\color{blue} ([0;1],0) \ar@[blue][lu] \ar@[blue][rd]} & & {\color{blue} ([1;2],0) \ar@[blue][ru] \ar@[blue][ld]} & \\
    & & {\color{blue} ([1;1],0)} & & }
  \end{equation*}
\end{exam}

Informally, each interval in $M_\phi$ determines a region in $\Omega_\phi$. The cone functor maps each interval to the formal limit over all intervals in its associated region. To make this rigorous, we define the cone functor $\kappa_\phi$ on objects to be
\begin{equation*}
 \kappa_\phi: x \mapsto \Pyr{\overline{\kappa}_\phi[x;x]} \in \widehat{ \Pyr{\Omega_\phi}},
\end{equation*}
and define it to send a morphism $x \leq x'$ to the following diagram defining a span in $\widehat{\Pyr{\Omega_\phi}}$
\begin{equation*}
 \xymatrixrowsep{.3pc}\xymatrixcolsep{.9pc} \xymatrix{ &  \kappa_\phi[x;x'] & & & &  \Pyr{\overline{\kappa}_\phi[x;x']} & \\
  & & & = & & & \\
 \kappa_\phi (x) \ar@{^{(}->}[ruu] & & \kappa_\phi(x')\ar@{_{(}->}[luu] & & \Pyr{\overline{\kappa}_\phi[x;x]} \ar@{^{(}->}[ruu] & &  \Pyr{\overline{\kappa}_\phi[x';x']}\, . \ar@{_{(}->}[luu] & }
\end{equation*}
Given a pair of composable morphisms $x \leq x' \leq x''$ in $M_\phi$, one has the following commutative diagram in $\widehat{\Pyr{\Omega_\phi}}^{\rm op}$
\begin{equation}
\label{KappaOplaxPush}
 \xymatrixrowsep{.9pc} \xymatrix{ \kappa_\phi(x') \ar@{^{(}->}[r] \ar@{^{(}->}[d] & \kappa_\phi[x';x''] \ar@{^{(}->}[d] \\
 \kappa_\phi[x;x'] \ar@{^{(}->}[r] & \kappa_\phi[x;x''] \, .}
\end{equation}
We define the corresponding component of the oplax structure,
\begin{equation*}
 K_\phi[x;x';x'']: \kappa_\phi[x';x''] \coprod_{\kappa_\phi(x')} \kappa_\phi[x;x'] \to \kappa_\phi[x;x''],
\end{equation*}
to be the universal morphism induced by the diagram in Eq.~\ref{KappaOplaxPush}.

\begin{lem}
 For each $\phi \in N_k(\Delta^{\rm op})$ the above data defines a normal oplax functor and hence, by Remark \ref{Rem:oplax}, a functor
 \begin{equation*}
  \kappa_\phi:\Pyr{M_\phi} \to \widehat{\Pyr{\Omega_\phi}}.
 \end{equation*}
\end{lem}
\begin{proof}
 From the diagram in Eq.~\ref{KappaOplaxPush} it is clear that $K_\phi[x;x;x'] = \id_{\kappa_\phi[x;x']}=K_\phi[x;x';x']$. Next, consider a triple of composable morphisms $w \leq x \leq y \leq z$ in $M_\phi$. One has the following diagram
 \begin{equation*}
 \xymatrixrowsep{.7pc} \xymatrixcolsep{1pc} \xymatrix{ & & & \kappa_\phi[w;z] & & & \\
 & & \kappa_\phi[w;y] \ar[ru] & & \kappa_\phi[x;z] \ar[lu] & & \\
 & \kappa_\phi[w;x] \ar[ru] & & \kappa_\phi[x;y] \ar[ru] \ar[lu] & & \kappa_\phi[y;z] \ar[lu] & \\
 \kappa_\phi(w) \ar[ru] & & \kappa_\phi(x) \ar[ru] \ar[lu] & & \kappa_\phi(y) \ar[ru] \ar[lu] & & \kappa_\phi(z) \, .\ar[lu] }
\end{equation*}
Then the following equation holds
\begin{equation*}
 K_\phi[w;y;z] \circ \left( \id_{\kappa_\phi[y;z]} \coprod_{\kappa_\phi(y)} K_\phi[w;x;y]\right)  = K_\phi[w;x;z] \circ \left( K_\phi[x;y;z] \coprod_{\kappa_\phi(x)} \id_{\kappa_\phi[w;x]} \right), 
\end{equation*}
since both sides of the equation are the universal morphism coming from the colimit of the lower two rows of the preceding diagram.
\end{proof}

The cone functor $\kappa_\phi$ is natural in the following sense.
\begin{lem}
\label{ConeNat}
 For each $\phi \in N_k(\Delta^{\rm op})$ and $\psi:[n] \to [k]$ the following diagram commutes
 \begin{equation*}
  \xymatrixrowsep{1.1pc} \xymatrix{ \Pyr{M_{\phi \psi}} \ar[d]_-{\Pyr{M(\psi)}} \ar[r]^-{\kappa_{\phi \psi}} & \widehat{\Pyr{\Omega_{\phi \psi}}} \ar[d]^-{\widehat{\Pyr{\Omega(\psi)}}} \\
  \Pyr{M_\phi} \ar[r]_-{\kappa_\phi} & \widehat{\Pyr{\Omega_\phi}}\, . }
 \end{equation*}
\end{lem}
\begin{proof}
It suffices to show that for each interval $[x;x']$ in $M_{\phi\psi}$,
\begin{equation*}
 \uparrow\, \Pyr{\Omega(\psi) \overline{\kappa}_{\phi \psi} [x;x']} = \Pyr{\overline{\kappa}_\phi[M(\psi)x;M(\psi)x']}.
\end{equation*}
This is straightforward, but notationally cumbersome, to verify.
\end{proof}

Finally, it is trivial to verify that $\kappa_\phi$ is vertically constant in the following sense.
\begin{lem}
 \label{ConeVert}
 For each $\phi \in N_k(\Delta^{\rm op})$ the following diagram commutes
 \begin{equation*}
  \xymatrixrowsep{.9pc} \xymatrix{ \Pyr{M_\phi} \ar[r]^-{\kappa_\phi} & \widehat{\Pyr{\Omega_\phi}} \\
  \Pyr{V_\phi} \ar@{^{(}->}[u] \ar[r]_-{\Pyr{\nu_\phi}} & \Pyr{\Omega_\phi}\, . \ar@{^{(}->}[u]}
 \end{equation*}
\end{lem}

\paragraph{The morphism $\beta$ is a symmetric monoidal lax functor.} The diagram in Eq.~\ref{ResRKE} defines a functor
\begin{equation*}
 R'_{(f,\phi)}: \Fun\left(\Singop{f(0)}\times \Pyr{\Omega_\phi}, \tSetfop\right) \to \Fun^{\rm cart}\left( \Cartop{f(0)}\times \Pyr{M_\phi}, \tSetfop \right).
\end{equation*}
We define $\beta(\theta,\gamma)$, for $((f,\phi),\theta,\gamma)\in \Un(\csmBialg)_k$, to be
\begin{equation*}
 \beta(\theta,\gamma):= R'_{(f,\phi)} \overline{\beta}(\theta,\gamma) \in \Unstr{\tSetfop}_k,
\end{equation*}
where $\Unstr{\tSetfop}$ is the sub simplicial set of the unstraightening of $\csmtSpan{\ctSetfop}$ from Proposition \ref{UnstrSimps}.

\begin{lem}
 The assignment $((f,\phi),\theta,\gamma) \mapsto \beta(\theta,\gamma)$ defines a morphism of simplicial sets
 \begin{equation*}
  \beta: \Un\left( \csmBialg\right) \to \Unstr{\tSetfop}.
 \end{equation*}
\end{lem}
\begin{proof}
 Fix $((f,\phi),\theta,\gamma) \in \Un\left( \csmBialg\right)$ and $\psi:[n] \to [k]$. The functor $\beta(\theta,\gamma)$ is cartesian by construction and Lemmas \ref{ConeVert} and \ref{MasterVert} imply that it is vertically constant.

 It remains to show that the functors $\psi^*\beta(\theta,\gamma)$ and $\beta(\psi^* \theta, \psi^* \gamma)$ agree, where 
 \begin{equation*}
  (\psi^* \theta)_i = (f_{\psi(i),i}, \phi_{\psi(i),i)})_* \theta_i \ \  {\rm and} \ \ (\psi^* \gamma)_i = (f_{\psi(i+1),i+1}, \phi_{\psi(i+1),i+1})_* \gamma_i. 
 \end{equation*}
 
Let $\overline{\Singop{f(0)}}$ denote the full subcategory of $\Cartop{f(0)}$ containing $\Singop{f(0)}$ as well as those $U \subset f(0)$ such that $U \subset f_{\psi(0),0}^{-1}(s)$ for some $s \in f\psi(0)$. Then by Lemma \ref{ConeNat} the following diagram commutes
\begin{equation*}
 \xymatrixrowsep{.9pc} \xymatrix{\Fun^{\rm cart}\left(\Cartop{f(0)}\times\Pyr{M_\phi},\tSetfop\right) \ar[r]^{\psi^*} & \Fun^{\rm cart}\left(\Cartop{f\psi(0)}\times\Pyr{M_{\phi\psi}},\tSetfop\right) \\
 \Fun\left( \overline{\Singop{f(0)}}\times \Pyr{\Omega_\phi}, \tSetfop\right) \ar[u]^-{R'_{(f,\phi)}} \ar[r]_-{\psi^*} &  \Fun \left(\Singop{f\psi(0)}\times \Pyr{\Omega_{\phi\psi}}, \tSetfop\right) \ar[u]_-{R'_{(f,\phi)}} }
\end{equation*}
To prove the Lemma it suffices to show that for each $s \in f \psi(0)$,
\begin{equation*}
 \overline{\beta}(\psi^* \theta, \psi^* \gamma)( s, -) = \coprod_{t\in f^{-1}_{\psi(0),0}(s)} \psi^*\overline{\beta}(\theta,\gamma)(t,-)
\end{equation*}
as normal oplax functors $\Omega_{\phi \psi} \nrightarrow \sp{\tSetfop}$. This is the case as, by definition, the functors $\theta_i$ are cocartesian.
\end{proof}

We have therefore constructed a morphism of fibrations
\begin{equation*}
  \xymatrixrowsep{.8pc}\xymatrixcolsep{.9pc}\xymatrix{\Un(\csmBialg)\ar[rr]^-{\beta} \ar[rd] & & \csmtSpan{\ctSetfop} \ar[ld] \\
  & \cFinDel & }
 \end{equation*}
 such that the image of $\underline{1}$ is $\Simplex{1,1}$. To show that $\beta$ endows the standard $(1,1)$-simplex $\Simplex{1,1}$ with the structure of a totally lax bialgebra it remains only to show that $\beta$ is a symmetric monoidal lax functor.

\begin{prop}
\label{TotLaxComb}
 The morphism of simplicial sets $\beta$ is a symmetric monoidal lax functor
 \begin{equation*}
  \beta: \csmBialg \rightsquigarrow \csmtSpan{\ctSetfop}.
 \end{equation*}
\end{prop}
\begin{proof}
 Fix $((f,\phi),\theta,\gamma)$ with $\phi:[m] \inmorL [n]$ inert. We must show that the induced functor
 \begin{equation*}
  \beta(\theta,\gamma)_{1,[1]}: \xymatrix{\Cartop{f(1)}\times \Pyr{\phi(1),1} \ar[r] & \Cartop{f(0)}\times \Pyr{M_\phi} \ar[r]^-{\beta} & \tSetfop}
 \end{equation*}
 is an equivalence in the quasicategory $\Sp_{f(1),\phi(1)}(\ctSetfop)$ (\cite{HTT} 3.2.5.2).
 
 As every inert morphism can be written as
 \begin{equation*}
  \phi:\xymatrixcolsep{1.2pc}\xymatrix{ [a] \vee \left( \displaystyle \bigvee_{i=1}^n [1] \right) \vee [b] & \ar@{>->}[l] \displaystyle \bigvee_{i=1}^n [1]\, ,}
 \end{equation*}
the poset $M_\phi$ is of the form
\begin{equation*}
  \includegraphics[width=11cm]{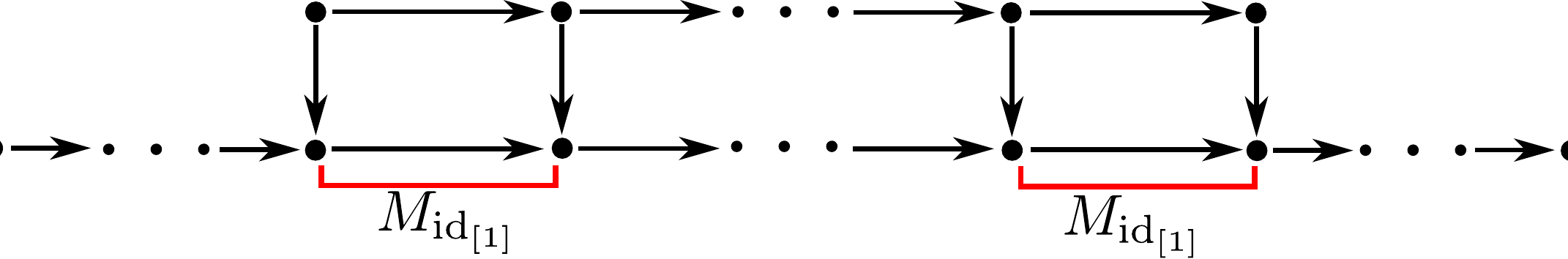}
 \end{equation*}
Let $(\theta(i),\gamma(i))$ denote the restriction of $(\theta,\gamma)$ to the $i$'th summand of $[n]$. Then for each $U \in \Cartop{f(1)}$, the functor $\beta(\theta,\gamma)_{1,[1]}(U,-)$ is the right Kan extension of a diagram of the following form in $\tSetfop$:
\begin{equation*}
  \includegraphics[height=2.25cm]{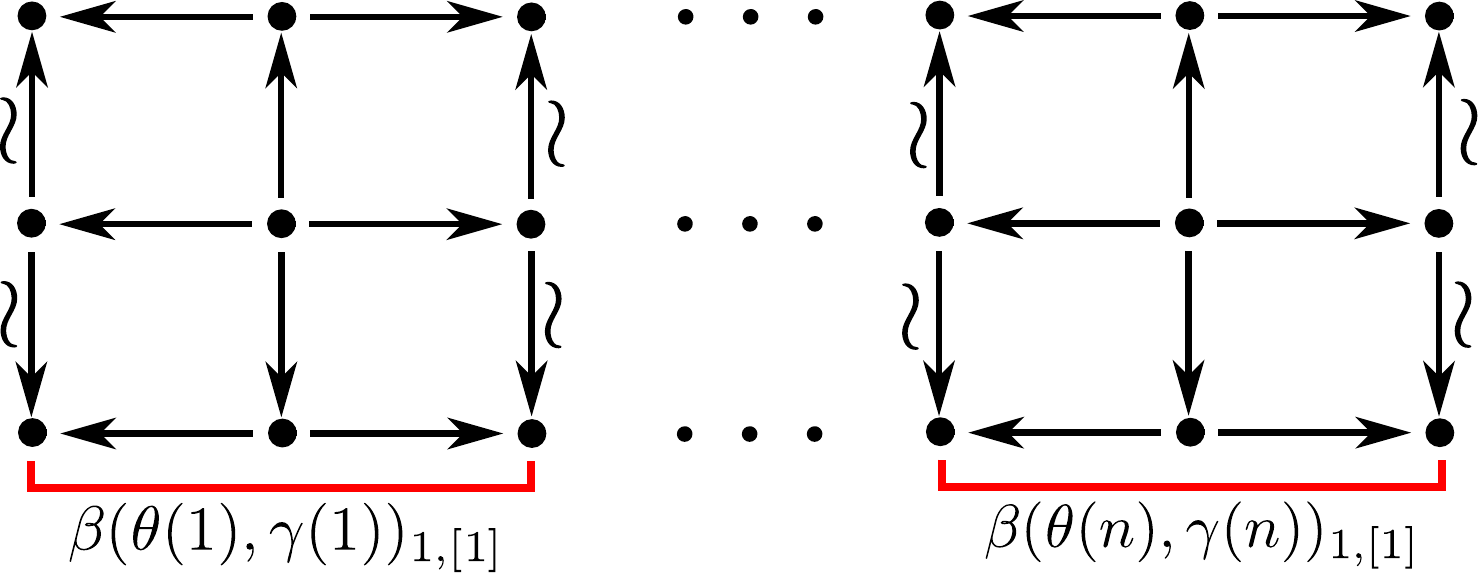}
 \end{equation*}
We conclude that $\beta(\theta,\gamma)_{1,[1]}$ is an equivalence if and only if $\beta(\theta(i),\gamma(i))_{1,[1]}$ is an equivalence for each $i$. 

It therefore suffices to consider the case when $\phi$ is the identity morphism. This is exactly the case considered above in Example \ref{MasterOne}. For each $U \in \Cartop{f(1)}$, the functor $\beta(\theta,\gamma)_{1,[1]}(U,-)$ is 
\begin{equation*}
 \xymatrixrowsep{.85pc} \xymatrixcolsep{1.2pc} \xymatrix{ \beta(X_0) \ar[d]^-\wr \ar[r] & \beta(p_0) \coprod_{\beta(Y)} \beta(p_1) \ar[d]^-\wr & \beta(X_1) \ar[l] \ar[d]^-\wr \\
 \beta(\gamma_0) \ar[r] & \beta(\gamma_0 p_0) \coprod_{\beta(\gamma)} \beta(\gamma_1 p_1) & \beta(\gamma_1) \ar[l] \\
 \beta(X_0') \ar[u]_-\wr \ar[r] & \beta(p_0') \coprod_{\beta(Y')} \beta(p_1') \ar[u]_-\wr & \beta(X_1') \ar[l] \ar[u]_-\wr }
\end{equation*}
and so $\beta(\theta,\gamma)_{1,[1]}$ is an equivalence.
\end{proof}

\subsection{The totally lax bialgebra structure inherited from \texorpdfstring{$\Simplex{1,1}$}{the standard (1,1)-simplex}}
\label{Sec:intospacesBi}

Let $\cX \in \cC_{\Simp^2}$ be a bisimplicial object in an $\infty$-category having finite limits. We can now show how the object $\cX_{1,1}$ inherits the structure of a totally lax bialgebra from the one endowed upon $\Simplex{1,1}$ in Proposition \ref{TotLaxComb}. 

Recall that any bisimplicial object $\cX$ defines a finite limit preserving functor $\cX:\ctSetfop \to \cC$  by right Kan extension,
\begin{equation*}
 \xymatrixrowsep{1.1pc} \xymatrix{ \ctSetfop \ar[r]^-\cX & \cC \\
 (\Simpop)^2 \ar@{^{(}->}[u] \ar[ru]_-\cX & }
\end{equation*}
Li-Bland has shown (\cite{LiBland} 4.1) that every finite limit preserving functor $F: \cD \to \cD'$ between $\infty$-categories having finite limits induces a symmetric monoidal functor $F: \csmtSpan{\cD} \to \csmtSpan{\cD'}$. One therefore has a symmetric monoidal functor
\begin{equation*}
 \cX: \csmtSpan{\ctSetfop} \to \csmtSpan{\cC}.
\end{equation*}

Then the following is a direct corollary of Proposition \ref{TotLaxComb}.
 \begin{thm}
 \label{TotLaxMain}
  Let $\cC$ be an $\infty$-category with finite limits and $\cX \in \cC_{\Simp^2}$ be a bisimplicial object. Then the composite
  \begin{equation*}
   \beta_\cX: \xymatrix{ \csmBialg \ar@{~>}[r]^-\beta & \csmtSpan{\ctSetfop} \ar[r]^-{\cX} & \csmtSpan{\cC}}
  \end{equation*}
is a symmetric monoidal lax functor which endows $\cX_{1,1}$ with the structure of a totally lax bialgebra.
 \end{thm}

\section{The universal Hall bialgebra}
\label{Sec:SegbiAssoc}

Having equipped the object of $(1,1)$-simplices $\cX_{1,1}$ in a bisimplicial object $\cX \in \cC_{\Simp^2}$ with a totally lax bialgebra structure in Theorem \ref{TotLaxMain}, it is now time to show that the double $2$-Segal condition enforces the (co)associativity of the (co)product. This will provide our definition of the universal Hall bialgebra of a double $2$-Segal object.

By Definition \ref{Defn:2Seg} a simplicial object $\cY \in \cC_\Simp$ is $2$-Segal if and only if the composite
 \begin{equation*}
  \alpha_\cY: \xymatrix{\csmAlg \ar@{~>}[r]^-\alpha & \csmtSpan{\csSetfop} \ar[r]^-{\cY} & \csmtSpan\cC}
 \end{equation*}
is a symmetric monoidal functor. We leverage this to deduce the relation between the double $2$-Segal and (co)associativity by proving that the lax algebra structure on $\cX_{1,1}$ is induced by the simplicial object $\cX_{\bullet,1}$ and the lax coalgebra structure is induced by $\cX_{1,\bullet}$.

 \begin{lem}
 \label{BetaAlphaComp}
  The totally lax bialgebra structure on $\Simplex{1,1}$ is compatible with the lax algebra and coalgebra structures on $\Simplex 1$. That is, the following diagram commutes
  \begin{equation*}
  \xymatrixrowsep{.9pc} \xymatrix{\csmAlg\ar@{~>}[r]^-{\alpha} \ar[d]_-{\iota_a} & \csmtSpan{\csSetfop} \ar[d]^-{(-)^h} \\
  \csmBialg \ar@{~>}[r]_-{\beta} & \csmtSpan{\ctSetfop} \\
  \csmCoalg \ar[u]^-{\iota_c} \ar@{~>}[r]_-\chi  & \csmtSpan{\csSetfop} \ar[u]_-{(-)^v} }
 \end{equation*}
 \end{lem}
\begin{proof}
We shall only prove that the top square commutes, as the proof for the bottom square is essentially identical.

 Fix $((f,\phi),\theta) \in \Un(\csmAlg)_k$. We must show that $\alpha(\theta)^h = \beta(\iota_a(\theta))$, where $\iota_a(\theta)$ is composed of the functors
 \begin{equation*}
  \xymatrix{\Cart{f(i)} \times \Pyrnc{\phi(i)} \ar[r] & \Cart{f(i)} \times \phi(i) \ar[r]& \Cart{f(0)} \times \phi(0) \ar[r]^-{\theta} & \Alg}
 \end{equation*}
and identity natural transformations between them. 

From the natural transformation $\Pyr{\bullet} \Rightarrow [\bullet]$ one has a functor $\Omega_\phi \to M_\phi$, and hence a functor
\begin{equation*}
 \rho_\phi: \xymatrixcolsep{1.2pc} \xymatrix{ \Pyr{\Omega_\phi} \ar[r]& \Pyr{M_\phi} \ar[r]^-{\Pyr{p_\phi}} & \Pyr{\phi(0)}}.
\end{equation*}
Consider the diagram
\begin{equation*}
 \xymatrixrowsep{.9pc} \xymatrix{ \Pyr{\phi(0)} & \ar[l]_-{\Pyr{p_\phi}} \Wedge{M_\phi} \ar[d]^-{\kappa_\phi} \\
 \Pyr{\Omega_\phi} \ar[u]^-{\rho_\phi} \ar@{^{(}->}[r]_-{\uparrow} & \widehat{\Pyr{\Omega_\phi}} \ar[lu]_-{\hat{\rho}_\phi}\, . }
\end{equation*}
The bottom triangle commutes by definition, and the upper triangle commutes since for any interval $[(a,b);(a',b')]$ in $M_\phi$,
\begin{equation*}
 [\phi_{b,0}(a); \phi_{b',0}(a')] = \min \left\{ [\phi_{c,0}(i); \phi_{c',0}(i')] \, | \, b \leq c \leq c' \leq b', \, \phi_{b,c}(a) \leq i, \, \phi_{b,c'}(a')\leq i' \right\}.
\end{equation*}
It follows that the following diagram commutes
\begin{equation*}
 \xymatrixrowsep{.9pc} \xymatrix{ \Fun\left(\Singop{f(0)}\times \Pyr{\phi(0)},\tSetfop \right) \ar[r]^-{(\Pyr{p_\phi})^*} \ar[d]_-{\rho_\phi^*} & \Fun\left(\Singop{f(0)}\times\Wedge{M_\phi},\tSetfop\right) \\
 \Fun\left(\Singop{f(0)}\times \Pyr{\Omega_\phi}, \tSetfop\right) \ar[r]^-{\hat{(- )}} & \Fun\left(\Singop{f(0)}\times\widehat{\Pyr{\Omega_\phi}},\tSetfop\right) \ar[u]_-{\kappa_\phi^*} }
\end{equation*}
since
\begin{equation*}
 \kappa_\phi^* \hat{(-)} \rho_\phi^* = \kappa_\phi^* \hat{(- )} \uparrow^* \hat{\rho}_\phi^* = \kappa^*_\phi \hat{\rho}_\phi^* = (\Pyr{p_\phi})^*.
\end{equation*}

To show that $\alpha(\theta)^h = \beta(\iota_a(\theta))$ it therefore suffices to show that 
\begin{equation*}
 \xymatrixrowsep{.9pc} \xymatrix{ \Singop{f(0)}\times \Omega_\phi \ar[r]^-{\overline{\beta}(\iota_a(\theta))} \ar[d] & \Sp{\tSetfop} \\
 \Singop{f(0)}\times\phi(0) \ar[r]_-{\overline{\alpha}(\theta)} & \Sp{\sSetfop} \ar[u]_-{(-)^h} }
\end{equation*}
where the horizontal maps are normal oplax functors. The diagram commutes on objects, as given an object $(s,[i;j],b)$ one has
\begin{equation*}
 \iota_a(\theta)_b(s,[i;j]) \cdot \Simplex{1,1} = \theta(\phi_{b,0}(i))\cdot \Simplex{1,1}.
\end{equation*}
They agree on morphisms as for a morphism $(s,[i;j],b) \to (s,[i';j'],b')$, the image under $\iota_a(\theta)_{b'}$ of the diagram in Eq.~\ref{morphm1} is
\begin{equation*}
 \xymatrixrowsep{.9pc} \xymatrix{ \theta(\phi_{b,0}(i)) \ar@{=}[d] \ar[r] & \theta(\phi_{b',0}(i')) \ar@{=}[d] \\
 \theta(\phi_{b,0}(i)) \ar[r] & \theta(\phi_{b',0}(i'))\, .}
\end{equation*}
Similarly, for a pair of composable morphisms $(s,[i;j],b) \to (s,[i';j'],b')\to(s,[i'';j''],b'')$ the image under $\iota_a(\theta)_{b''}$ of the diagram in Eq.~\ref{morphm2} is
\begin{equation*}
 \xymatrixrowsep{.6pc} \xymatrix{ \theta(\phi_{b,0}(i)) \ar@{=}[d] \ar[r] & \theta(\phi_{b',0}(i')) \ar@{=}[d] \ar[r] & \theta(\phi_{b'',0}(i'')) \ar@{=}[d] \\
 \theta(\phi_{b,0}(i)) \ar@{=}[d] \ar[r] & \theta(\phi_{b',0}(i')) \ar@{=}[d] \ar[r] & \theta(\phi_{b'',0}(i'')) \ar@{=}[d] \\
 \theta(\phi_{b,0}(i))  \ar[r] & \theta(\phi_{b',0}(i'))  \ar[r] & \theta(\phi_{b'',0}(i''))  }
\end{equation*}
showing that the diagram commutes on the level of oplax structures.
\end{proof}

\begin{thm}
\label{UnivBialg}
 Let $\cX\in \cC_{\Simp^2}$ be a bisimplicial object in an $\infty$-category having finite limits. If $\cX$ is a double $2$-Segal object then the symmetric monoidal lax functor $\beta_\cX$ endows $\cX_{1,1}$ with a lax bialgebra structure.
\end{thm}
\begin{defn}
\label{UnivBialgdefn}
 The {\em universal Hall bialgebra} of a double $2$-Segal object $\cX$ is $\cX_{1,1}$ equipped with the lax bialgebra structure of Theorem \ref{UnivBialg}
\end{defn}
\begin{proof}
 By Lemma \ref{BetaAlphaComp} the following diagram commutes for any bisimplicial object $\cX$
 \begin{equation*}
 \xymatrixrowsep{.9pc} \xymatrix{ \csmAlg \ar[d]_-{\iota_a} \ar@{~>}[rd]^-{\alpha_{\cX_{\bullet,1}}} & \\
 \csmBialg \ar@{~>}[r]^-{\beta_\cX} & \csmtSpan{\cC} \\
 \csmCoalg \ar[u]^-{\iota_c} \ar@{~>}[ru]_-{\chi_{\cX_{1,\bullet}}} & }
\end{equation*}
By Definition \ref{Defn:2Seg}, a simplicial object $\cY \in \cC_\Simp$ is $2$-Segal if and only if $\alpha_{\cY}$ and $\chi_{\cY}$ are symmetric monoidal functors. Therefore $\beta_\cX$ endows $\cX_{1,1}$ with a lax bialgebra structure when $\cX_{\bullet,1}$ and $\cX_{1,\bullet}$ are $2$-objects.
\end{proof}

 \bibliographystyle{plain}
\bibliography{refs}

\begin{thebibliography}{10}

\bibitem{barwick2015exact}
Clark Barwick.
\newblock On exact $\infty$-categories and the {T}heorem of the {H}eart.
\newblock {\em Compositio Mathematica}, 151(11):2160--2186, 2015.

\bibitem{WiTWald}
Julia~E. Bergner, Ang\'elica~M. Osorno, Viktoriya Ozornova, Martina Rovelli,
  and Claudia~I. Scheimbauer.
\newblock 2-{S}egal sets and the {W}aldhausen construction.
\newblock {\em arxiv preprint arXiv:1609.02853}, 2016.

\bibitem{D15}
Tobias Dyckerhoff.
\newblock Higher categorical aspects of {H}all algebras.
\newblock {\em arXiv preprint arXiv:1505.06940}, 2015.

\bibitem{DK12}
Tobias Dyckerhoff and Mikhail Kapranov.
\newblock Higher {S}egal spaces {I}.
\newblock {\em arXiv preprint arXiv:1212.3563}, 2012.

\bibitem{Errington}
David~Lindsay Errington.
\newblock {\em Twisted Systems}.
\newblock PhD thesis, Imperial College London, 1999.

\bibitem{KockI}
Imma G{\'a}lvez-Carrillo, Joachim Kock, and Andrew Tonks.
\newblock Decomposition spaces, incidence algebras and {M}\"obius inversion
  {I}: basic theory.
\newblock {\em arXiv preprint arXiv:1512.07573}, 2015.

\bibitem{Kockres}
Imma G{\'a}lvez-Carrillo, Joachim Kock, and Andrew Tonks.
\newblock Decomposition spaces and restriction species.
\newblock {\em arXiv preprint arXiv:1708.02570}, 2017.

\bibitem{Green}
James~A Green.
\newblock Hall algebras, hereditary algebras and quantum groups.
\newblock {\em Inventiones mathematicae}, 120(1):361--377, 1995.

\bibitem{RuneSpans}
Rune Haugseng.
\newblock Iterated spans and ``classical" topological field theories.
\newblock {\em arXiv preprint arXiv:1409.0837}, 2014.

\bibitem{JoyalJPAA}
Andr{\'e} Joyal.
\newblock Quasi-categories and {K}an complexes.
\newblock {\em Journal of Pure and Applied Algebra}, 175(1):207--222, 2002.

\bibitem{joyal2008theory}
Andr{\'e} Joyal.
\newblock The theory of quasi-categories and its applications.
\newblock {\em unpublished. Available at
  \url{http://mat.uab.cat/~kock/crm/hocat/advanced-course/Quadern45-2.pdf}},
  2008.

\bibitem{joyce2007configurations}
Dominic Joyce.
\newblock Configurations in abelian categories. {II}. {R}ingel--{H}all
  algebras.
\newblock {\em Advances in Mathematics}, 210(2):635--706, 2007.

\bibitem{Lack}
Stephen Lack.
\newblock Composing {PROP}s.
\newblock {\em Theory and Applications of Categories}, 13(9):147--163, 2004.

\bibitem{LiBland}
David Li-Bland.
\newblock The stack of higher internal categories and stacks of iterated spans.
\newblock {\em arXiv preprint arXiv:1506.08870}, 2015.

\bibitem{LurieHA}
Jacob Lurie.
\newblock {\em Higher algebra}.
\newblock \url{http://www.math.harvard.edu/~lurie/}.
\newblock [Online; accessed August 2017].

\bibitem{HTT}
Jacob Lurie.
\newblock {\em Higher Topos Theory}.
\newblock Princeton University Press, 2009.

\bibitem{Lurieinftwo}
Jacob Lurie.
\newblock $(\infty, 2)$-categories and the {G}oodwillie {C}alculus {I}.
\newblock {\em arXiv preprint arXiv:0905.0462}, 2009.

\bibitem{lusztig1991quivers}
George Lusztig.
\newblock Quivers, perverse sheaves, and quantized enveloping algebras.
\newblock {\em Journal of the american mathematical society}, 4(2):365--421,
  1991.

\bibitem{penney2017bimon}
Mark~D Penney.
\newblock The {H}all bimonoidal categories.
\newblock {\em in preparation}, 2017.

\bibitem{penney2017simplicial}
Mark~D Penney.
\newblock Simplicial spaces, lax algebras and the 2-{S}egal condition.
\newblock {\em arXiv preprint arXiv:1710.02742}, 2017.

\bibitem{Pirash02}
Teimuraz Pirashvili.
\newblock On the {PROP} corresponding to bialgebras.
\newblock {\em Cahiers de Topologie et G{\'e}om{\'e}trie Diff{\'e}rentielle
  Cat{\'e}goriques}, 43(3):221--239, 2002.

\bibitem{Ringel}
Claus~Michael Ringel.
\newblock Hall algebras and quantum groups.
\newblock {\em Inventiones mathematicae}, 101(1):583--591, 1990.

\bibitem{toen2006derived}
Bertrand To{\"e}n et~al.
\newblock Derived {H}all algebras.
\newblock {\em Duke Mathematical Journal}, 135(3):587--615, 2006.

\bibitem{Waldhausen}
Friedhelm Waldhausen.
\newblock Algebraic {K}-theory of spaces.
\newblock {\em Algebraic and geometric topology (New Brunswick, NJ, 1983)},
  1126:318--419, 1985.

\end{thebibliography}

 \end{document}